\documentclass[10.6pt, reqno]{amsart}
\usepackage{xcolor}

\usepackage{tikz}

\usepackage{graphicx} 
\usetikzlibrary{decorations.markings}

\definecolor{mypink}{HTML}{EA909D}
\definecolor{myblue}{HTML}{1D70B8}
\usepackage[citebordercolor=green]{hyperref}
\usepackage{relsize}
\usepackage{tikz}
\usepackage{tabularx}
\usepackage{caption}
\usepackage{graphicx}
\usepackage{cancel}
\usepackage{subfigure}
\usepackage{slashed}
\usepackage{amsfonts}
\usepackage{amssymb}
\usepackage{accents}
\usepackage{amsmath}
\usepackage{amsxtra}
\usepackage{epstopdf}
\usepackage{latexsym}
\usepackage{mathrsfs}
\usepackage{leftidx}
\usepackage{tensor}
\usepackage{enumitem}
\usepackage{setspace}

\usepackage[dvipsnames]{xcolor}

\newtheorem{theorem}{Theorem}[section]

\newtheorem{lemma}[theorem]{Lemma}
\newtheorem{corollary}[theorem]{Corollary}
\newtheorem{proposition}[theorem]{Proposition}
\newtheorem{definition}[theorem]{Definition}

\theoremstyle{remark}
\newtheorem{remark}[theorem]{Remark}

\numberwithin{equation}{section}

\def\sC{\mathscr{C}}

\def\sta#1#2{\stackrel{#1}{#2}}
\def\stc#1{\sta{\circ}{#1}}

\def\rp#1{^{\!(#1)}}

\def\bp{\boldsymbol{\partial}}

\def\T{\mathcal{T}}

\def\bb{{\mathbf{b}}}

\def\Er{\mbox{Er}}

\def\sn{{\slashed{\nabla}}}

\def\sQ{\mathscr{Q}}

\def\eh{\hat{\eta}}

\def\zb{{\underline{\zeta}}}

\def\bT{{\textbf{T}}}

\def\bR{{\textbf{R}}}

\def\bd{{\textbf{D}}}
\def\ti{\tilde}

\def\bg{\mathbf{g}}

\def\hk{{\hat{k}}}

\def\beaa{\begin{eqnarray*}}
\def\eeaa{\end{eqnarray*}}

\def\ba{\begin{array}}
\def\ea{\end{array}}

\def\be#1{\begin{equation} \label{#1}}
\def \eeq{\end{equation}}

\newcommand{\nn}{\nonumber}

\def\l{\langle}
\def\r{\rangle}

\def\nn{\nonumber}
\def\S{{\mathcal S}}
\def\fS{\mathfrak{S}}

\def\ud#1{\underline{#1}}

\def\S2{{\mathbb S}^2}

\def\bB{\mathbf{B}}

\def\E{{\mathcal E}}

\def\K{{\mathcal{K}}}

\def\ze{{\zeta}}

\def\Lie{{\mathcal L}}

\def\tr{\mbox{tr}}

\def\H{{\mathcal H}}

\def\N{{\mathcal N}}

\def\B{{\mathcal B}}

\def\R{{\mathcal R}}
\def\c{\cdot}

\def\a{\alpha}
\def\b{\beta}

\def\ep{{\epsilon}}
\def\ve{{{\textbf{$\varepsilon$}}}}
\def\l{\langle}
\def\r{\rangle}
\def\ga{\gamma}
\def\Ga{\Gamma}

\def\O{\mathcal{O}}

\def\p{\partial}

\def\nab{\nabla}
\def\hb{{\ud h}}

\def\Lb{{\underline{L}}}

\def\div{\mbox{\,div\,}}

\def\tr{\mbox{tr}}
\def\Tr{\mbox{Tr}}

\def\tir{{\tilde r}}

\def\f14{\frac{1}{4}}
\def\f12{{\frac{1}{2}}}

\def\t1a{t^{-\frac{1}{a}}}

\def\bm{{\bf m}}

\def\sD{\slashed{\Delta}}

\def\sn{{\slashed{\nabla}}}

\def\eh{\hat{\eta}}

\def\zb{{\underline{\zeta}}}

\def\bT{{\emph{\bf{T}}}}

\def\bR{{\emph{\bf{R}}}}

\def\bd{{\emph{\bf{D}}}}
\def\ti{\tilde}

\def\hk{{\hat{k}}}

\def\beaa{\begin{eqnarray*}}
\def\eeaa{\end{eqnarray*}}

\def\ba{\begin{array}}
\def\ea{\end{array}}
\def\be#1{\begin{equation} \label{#1}}
\def \eeq{\end{equation}}
\def\nn{\nonumber}

\def\l{\langle}
\def\r{\rangle}

\def\nn{\nonumber}
\def\S{{\mathcal S}}

\def\S2{{\mathbb S}^2}

\def\E{{\mathcal E}}

\def\ze{{\zeta}}

\def\Lb{\underline{L}}
\def\tr{\mbox{tr}}

\def\H{{\mathcal H}}

\def\B{{\mathcal B}}

\def\R{{\mathcal R}}
\def\c{\cdot}

\def\a{\alpha}
\def\b{\beta}

\def\l{\langle}
\def\r{\rangle}
\def\ga{\gamma}

\def\Ga{\Gamma}
\def\la{\lambda}

\def\p{\partial}

\def\nab{\nabla}

\def\Lb{{\underline{L}}}

\def\div{\mbox{\,div\,}}

\def\tr{\mbox{tr}}
\def\Tr{\mbox{Tr}}

\def\tir{{\tilde r}}

\def\f14{\frac{1}{4}}
\def\f12{{\frac{1}{2}}}

\def\t1a{t^{-\frac{1}{a}}}

\def\bm{{\bf m}}

\def\sD{\slashed{\Delta}}

\newcommand{\bea}{\begin{eqnarray}}
\newcommand{\eea}{\end{eqnarray}}

\def\nn{\nonumber}

\newcommand{\chih}{\hat{\chi}}
\newcommand{\chib}{\underline{\chi}}

\newcommand{\chibh}{\underline{\hat{\chi}}\,}
\newcommand{\les}{\lesssim}
\newcommand{\ges}{\gtrsim}

\def\bN{{\mathbf{N}}}

\def\S{\mathcal{S}}

\def\vs{\varsigma}

\def\sG{{\mathscr{G}}}
\def\sY{\mathscr{Y}}

\def\sQ{{\mathscr{Q}}}

\def\ud#1{\underline{#1}}

\def\fw{\mathfrak{w}}
\def\be{{(e)}}

\def\scg{\stc{\bg}}

\begin{document}
\title{Inevitable shock formation for $3$-D compressible Euler flows}
\author{Sergiu Klainerman}
\address{Department of Mathematics, Princeton University}
\email{seri@math.princeton.edu}

\author{Qian Wang}
\address{Mathematical Institute, University of Oxford}
\email{qian.wang@maths.ox.ac.uk}

\author{Shiwu Yang}
\address{Beijing International Center for Mathematical Research, Peking University}
\email{shiwuyang@math.pku.edu.cn}

\author{Pin Yu}
\address{Department of Mathematical Sciences and New Cornerstone Science Laboratory, Tsinghua University
}
\email{yupin@mail.tsinghua.edu.cn}
  \date{\today}
 
\begin{abstract}
We prove that solutions arising from smooth, sufficiently small, compactly supported perturbations of non-vacuum constant states in three-dimensional irrotational compressible flow must blow up in finite time, without any symmetry assumptions or other restrictions on the initial data. Moreover, we prove that shock inevitably forms at the boundary of the maximal Cauchy development, and that its formation time agrees, in the small-data asymptotic
regime, with the lifespan predicted by the radiation-field analysis.
\end{abstract}
 \maketitle

\section{Introduction}
\subsection{Motivation}\label{mot}

   A defining feature of the theory of quasilinear hyperbolic systems is that
   smooth initial data may eventually lead to singularity formation. This
   phenomenon is well understood in one space dimension, beginning with the
   scalar Burgers' equation\footnote{Research in this direction, related to the
   propagation of one-dimensional sound waves, was initiated by Monge, Poisson,
   Stokes, Riemann, and others.}.  The blow--up mechanism in the scalar case
   was extended to $2\times2$ systems, and in particular to second-order wave
   equations satisfying the genuine nonlinearity condition, in the works of
   Oleinik \cite{Oleinik} and Lax \cite{Lax1,Lax2}.  Klainerman and Majda
   \cite{K-Ma} treated a larger class of equations, including the nonlinear
   vibrating string equations, for which genuine nonlinearity is violated. General blow--up results for genuinely nonlinear\footnote{John's result was
   extended to more general systems by Liu \cite{TPLiu}. See also \cite{Chris_Perez}}. $n\times n$ systems are
   due to John \cite{FJohn}.   He showed, in the context of such systems,   that  compactly  supported  data \footnote{ The phenomenon of singularity formation in 1D  is  easiest to detect  for compactly supported
   data, since such data necessarily contain a compression region, that is, a
   region in which characteristics of the same family, issued from the initial  data,
   point toward one another.} of size proportional to
   $\ve$ lead to a shock-wave singularity in time $O(\ve^{-1})$. 
  In all these $1$-D  results, shock formation is tied to
   the collapse of the characteristic flow.

   In higher dimensions the situation is far more complicated, both because of
   the complexity of the equations and because of dispersion, a phenomenon
   typical of scalar quasilinear wave equations, including those arising from
   irrotational compressible Euler flow.  Dispersion may significantly delay,
   or even altogether prevent\footnote{As in the case of  equations satisfying the null
   condition.}, singularity formation for small, compactly supported
   perturbations of trivial data.

Foundational blow-up results for nonlinear and quasilinear wave equations are due to F. John \cite{FJohn1, FJohn2}. Related finite-time breakdown results for multidimensional conservation laws were obtained by Sideris \cite{Sideris1}\begin{footnote}{This is a result under large initial disturbance.}\end{footnote}, who subsequently treated the physically important three-dimensional compressible Euler equations \cite{Sideris} \begin{footnote}{This paper assumes initial compression, which also applies to small data.}\end{footnote}. These results establish   non-existence of global smooth solutions but do not give a detailed geometric description of shock formation.


     F. John \cite{FJohn3} gave the first example of a simple
     three-dimensional quasilinear wave equation,
    \[
      \partial_t^2 u-\triangle u=\partial_t\big(F(\partial_tu)\big),
    \]
    for which a complete analysis could be carried out.  He showed that all
    spherically symmetric solutions with arbitrarily small, compactly
    supported initial data lead to future shock type  singularities.

  More precisely, he showed that all spherically symmetric solutions of this
  equation, under the genuine-nonlinearity-type condition
  $F(0)=0$, $F'(0)\neq0$, and corresponding to smooth, sufficiently small data
  of size $\ve>0$,
  \[
    u(0,x)=\ve u_0(x),\qquad
    \partial_tu(0,x)=\ve u_1(x),
  \]
  supported in $\{|x|\le R_0\}$, must blow up after a finite time $T$.  In fact,
  for given $F,u_0,u_1,R_0$, there  exist a constant $\ve_0$ and a function
  $A(\ve)$, depending only on $F,u_0,u_1, R_0$, such that
  \begin{itemize}
  \item[a.)] The extinction (or lifespan)  time $T$ satisfies the upper bound\footnote{A lower
  bound was derived by John and Klainerman in \cite{John-K}.  See also
  \cite{K1} for a proof of the same result using the vectorfield method.}, for
  all $0<\ve\le\ve_0$,
  \beaa
  T< \exp{\Big(\frac{A(\ve)}{\ve}\Big)}.
  \eeaa

  \item[b.)] $A(\ve)$ is bounded independently of $\ve$, i.e.
  \beaa
  A=\limsup_{\ve\to 0} A(\ve) <\infty.
  \eeaa
  \item[c.)] Singularity formation, as in the case of the Burgers' equation, is
  due to the collapse of the separation between neighboring characteristics: in particular, the second derivatives of $u$
  blow up while the first derivatives remain bounded.

  \end{itemize}
Under radial symmetry assumption, results in 2-D for compressible Euler system were derived in \cite{Alin_4, Alinhac}, while the 3-D irrotational case was treated by Yin \cite{Yin}. The irrotational condition is needed to recast the compressible
  Euler equations (CEE) as a quasilinear wave equation of the form,
  \bea
  \label{eq:waveeq-intro}
  \bg^{\a\b}(\bp \phi)  \partial_\a \partial_ \b \phi=0.
  \eea
  Here $\bg$ is a specific Lorentzian metric, depending on the first derivatives
  of $\phi$, whose  precise form is recalled in Section \ref{setup}, see   \eqref{wave}.

  The above results are to be contrasted with the global existence results for
  systems of wave equations satisfying the null condition, in which case
  $T(\ve)=\infty$ for $0<\ve\le\ve_0$ sufficiently small.  The development of the
  theory of quasilinear wave equations  can be organized around two contrasting regimes:

    \begin{enumerate} 
    \item[I.] Equations motivated by General Relativity and other geometric
     theories, for which the null condition plays a prominent role and thus  small-data global existence is expected. 
     \item[II.] Equations motivated by continuum mechanics, such as  CEE,  for which small, compactly supported
      initial data are expected to blow up in dimensions 2 and 3.\footnote{In dimension
      2, the exponential bound in b) above is to be replaced by a quadratic
      expression in $\frac{A(\ve)}{\ve}$.  The first such results were proved
      by Alinhac \cite{Alin_4}-\cite{Alin2}.}

      \end{enumerate}
       In the second case, which is the focus of this paper, it is natural to
       expect that the symmetry assumptions can be dispensed with. In
      other words,  for equations  of the form \eqref{eq:waveeq-intro}, satisfying an appropriate genuine-nonlinearity condition\footnote{Some genuine-nonlinearity-type
       condition is needed to exclude equations satisfying the null condition. One can also consider equations with an additional  semi-linear term in $\bp \phi$.}, 
       any small,
       compactly supported perturbation of constant-state initial data should
       develop a finite-time singularity of John's type, namely one satisfying
       the properties a), b), and c) above.
        In the specific case of irrotational CEE, the natural conjecture, which we verify in this paper, is
        that all smooth, compactly supported, sufficiently small perturbations
        of a (non-vacuum) constant state develop finite-time shock-wave singularities of
        John's type at the boundary of the maximal Cauchy development.  A
        broader conjecture should include other important equations of
        continuum mechanics, where compression plays an important role, in
        dimensions $2$ and $3$ where dispersion is not sufficiently strong to
        counteract compression. It is also natural to conjecture an analogous result for the general full Euler system with non-trivial vorticity and entropy, the proof of which is a work in progress,  see \cite{Wangnotes}.

    We now briefly describe the most  relevant developments after
    \cite{FJohn3}.

    \begin{enumerate}
    \item In \cite{FJohn4}, F. John gave an important improvement of the almost
    global existence result of John--Klainerman \cite{John-K} by deriving a
    lower bound for $A(\ve)$ in terms of the initial data; see also the
    follow-up paper \cite{FJohn5}.  A similar result was derived by
    H\"ormander in \cite{Hormander_1}.
     \item The first results on singularity formation, for a general class of quasilinear wave equations,  without symmetry
     assumptions on the initial data are due to Alinhac \cite{Alin1,Alin2}. However, in addition to assuming small, compactly supported,  smooth data,  Alinhac imposed a non-degeneracy condition on the radiation profile of the initial data, which singles out the first blow-up point and ensures its non-degeneracy. In these works, the blow-up mechanism is tied to the degeneration of the characteristic foliation and does not require symmetry.
    
     \item In the seminal work \cite{shock_demetrios}, Christodoulou proved  the formation of shock for the $3$-D relativistic Euler  equations, without symmetry assumptions, assuming a compression condition on initial
      data constructed in an annulus surrounded by a non-vacuum constant state.
      His work introduced into the study of three-dimensional shock formation
      the global geometric techniques developed in connection with the proof of
      the nonlinear stability of Minkowski space \cite{CK}.  As in one
      dimension, the mechanism of shock formation is the collapse of the
      characteristic flow.  This requires a systematic analysis of the
      characteristics, that is, the null hypersurfaces associated with the
      acoustical metric, and of the effect of the Euler equations on their
      geometric properties.
      
   The similar result was later  proved for the irrotational isentropic compressible Euler equations  by Christodoulou-Miao \cite{Miao_thesis}.

     \item Christodoulou's results and methods led to a great deal of
     subsequent work in this direction, including:
     \begin{itemize}
     \item Speck's result on shock formation for the geometric wave equation
     in three dimensions \cite{Spck_shock_1}, 
\item Miao--Yu's result on shock formation along incoming acoustical cones for
some quasilinear wave equations satisfying the classical null condition
\cite{Pin-Shuang}, 
\item Luk--Speck's works on shock formation for the two- and
three-dimensional compressible Euler equations, including the non-irrotational
case \cite{Spck-luk_2} and the non-isentropic case \cite{Spck-luk_3}.
\end{itemize}
    \end{enumerate}
     One can also refer to the works of Buckmaster--Shkoller--Vicol \cite{Buck_2}-\cite{Buck_4},
where shocks with vorticity and entropy are constructed by perturbing a
Burgers' shock.

     These works all give results on shock formation  for data satisfying a
     specific compression condition. 
    The goal of this paper is to show that
   for compactly supported smooth perturbations, no restriction beyond smallness is necessary: as in $1$-D, and as in
     John's spherically symmetric three-dimensional model, the relevant
     compression is encoded indirectly by the compact support condition for the data.
    
\subsection{Set-up}\label{setup}
We consider the compressible Euler equations of $3$-space dimension for a perfect fluid under a barotropic equation of state, that is, the
pressure $p$ is a function of the density  $\rho:{\mathbb R}^{1+3}\rightarrow (0,\infty)$,
\begin{equation}\label{10.12.1.19}
p=p(\rho).
\end{equation}
 We fix outside of the Euclidean sphere $\{|x|=1\}$ at $t=0$ a constant  background  density $\bar\rho>0$. Define the normalized density
\begin{equation}\label{10.12.2.19}
\varrho=\ln (\rho/\bar \rho)
\end{equation}
 and the sound speed
 \begin{equation*}
  c=\sqrt{\frac{dp}{d\rho}}.
 \end{equation*}
For convenience, we assume the pressure function
\begin{equation*}
p(\rho)=A\rho^\ga,
\end{equation*}
where the constant $A>0$, and $\ga> 1$, which ensures the constant\begin{footnote}{ Here $c'(\rho)$ is the derivative with respect to $\rho$.}\end{footnote}   $\wp=c^{-1}c'+1=\frac{\ga+1}{2}> 1$. 

Let $v$ be the velocity of the compressible fluid, $v:{\mathbb R}^{3+1}\rightarrow {\mathbb R}^3$. At $t=0$, let $\{|x|\le 1\}$ be surrounded by the constant states $(\bar\rho, \bar v)$. Using the Galilean transform of coordinates $x^a\rightarrow x^a-{\bar v}^a t$, $v\rightarrow v-\bar v$
 we may assume without loss of generality that
\begin{equation}\label{cst}
\varrho=0,\quad c=\bar c =(A\ga)^\f12 (\bar\rho)^{\frac{\ga-1}{2}}>0,  \quad  v=0, \mbox{ for }|x|\ge 1, t=0.
\end{equation}

We define the acoustic metric $\bg$ as
\begin{equation}\label{metric}
\bg:=-c^2 dt\otimes dt+\sum_{a=1}^3(d x^a-v^adt)\otimes (dx^a-v^a dt).
\end{equation}
The inverse metric $\bg^{-1}$ can be written as
\begin{equation*}
\bg^{-1}=-\bT\otimes \bT+\Sigma_{a=1}^3 \p_a \otimes \p_a,
\end{equation*}
where $\bT$ is the future directed, time-like unit normal of $\Sigma_t$, i.e. the level set of $t$.  The components of  $\bg^{-1}$ will be denoted by
$\bg^{\a\b}$. \begin{footnote}{We adopt Einstein summation convention in this article. The range of  the indices of Greek letters such as
$\a,\b,\mu,\nu$ is $0,\cdots, 3$, and the range of the Latin letters $i,j,k,l,m,n,a,b$ is $1,2,3$, unless specified differently.}\end{footnote}
 Relative to the Cartesian coordinates, we have
\begin{equation*}
c\bT=\p_t+v^a \p_a.
\end{equation*}

With $\bB=c\bT$, we introduce the compressible Euler equations with (\ref{10.12.1.19}) for $\varrho$ and $v$,
\begin{equation}\label{4.23.1.19}
\left\{
\begin{array}{lll}
\bB\varrho=-\div v\\
\bB v^i=-c^2\delta^{ia} \p_a \varrho,
\end{array}
\right.
\end{equation}
where $\div v= \p_i v^i$ and  $\varrho$ is the normalized density function.

 Let $\tensor{\ud\ep}{_i^j^k}$, $i,j,k=1,2,3$, be the standard volume form on ${\mathbb R}^3$. We define the vorticity to be
$
\fw_i=\tensor{\ud\ep}{_i^j_k}\p_j v^k.
$  \begin{footnote}{ Here the indices of the tensor fields and differentiation are lifted and lowered by the Euclidean metric.}\end{footnote} 

In the isentropic irrotational case, $\fw=0$ for all $t$. The compressible Euler equations (\ref{4.23.1.19}) can be reduced to \begin{footnote}{See the derivation of (\ref{4.10.1.19}) and (\ref{4.10.2.19}) in Section \ref{Re_wave}}.\end{footnote}
\begin{align}
& \Box_\bg v^i= \sQ^i, \label{4.10.1.19}\\
& \Box_\bg \varrho= \sQ^0,\label{4.10.2.19}
\end{align}
where $\Box_\bg$ is the Laplace-Beltrami operator of the Lorentzian metric $\bg$, and the  quadratic terms are
\begin{align*}
&\sQ^i:=(\wp-2)\bg^{\a\b}\p_\a\varrho\p_\b v^i, \\
&\sQ^0:=-\wp\bg^{\a\b}\p_\a \varrho \p_\b \varrho-c^{-2}\bg^{\a\b}\p_\a v^i \p_\b v^i.
\end{align*}
Define $v=\nab\phi$, where $\phi$ is the wave potential with $\phi(0)=\phi_0.$ 
The $3$-D irrotational compressbile Euler equation takes the form for the wave potential $\phi$ 
\begin{equation}\label{wave}
\begin{split}
&\p_t^2 \phi+2\sum_{k=1}^3\p_k\phi\p_t \p_k\phi-c^2\Delta \phi+\sum_{i=1}^3\p_i\phi \p_k \phi \p_i \p_k \phi=0\\
&\phi(0)=\phi_0=\ve f, \quad \p_t\phi(0)=\phi_1=\ve g
\end{split}
\end{equation}
where $f, g$ are compactly supported within $\{|x|\le 1\}$.  By finite speed of propagation, the solution $\phi$ is trivial if $r\ge r_M$ with $r_M=\bar c t+1$ and $\bar c$ the sound speed of the constant state. 

We define  $\iota=-(\p_t \phi+\f12\sum_{i=1}^3|\p_i\phi|^2)$, and the enthalpy $w$ by
\begin{equation*}
w(\rho)=\int^\rho_{\bar\rho}\frac{p'(\tau)}{\tau} d\tau.
\end{equation*}
By Bernoulli's law,
\begin{equation}\label{Bernoulli}
   \p_t\phi+\f12|\nab \phi|^2+w(\rho)=0.
\end{equation}
With
$
w(\rho)=\iota,
$
\begin{equation}\label{12.29.1.25}
c^2(\rho(\iota))=\frac{w^{-1}(\iota)}{(w^{-1})'(\iota)}.
\end{equation}
Here $\iota$ depends only on $\p_t\phi$ and $|\nab\phi|^2$. Hence $c^2$  is a smooth function depending only on $\bp\phi$ where $\bp$ denotes both $\p$ and $\p_t$.
Thus the coefficient of (\ref{wave}) depends on $\bp\phi$ only.

In this paper, we assume the initial data is irrotational, i.e. $\fw(0)=0$. From the given data $(v_0, \rho_0)$, we derive  from (\ref{Bernoulli}) 
\begin{equation}\label{12.29.4.25}
    \phi_1=-\f12|\nab \phi_0|^2-w(\rho_0), \quad w(\rho_0)=A\frac{\ga}{\ga-1}(\rho_0^{\ga-1}-\bar\rho^{\ga-1}).
\end{equation}
If $\phi_1=0$ and $\phi_0=0$, it clearly implies that $v(0)=0=\varrho(0)$. 

We fix the convention that $\l t\r=t+1$. Throughout this paper, a universal constant may depend only on $f$ and $g$  in  (\ref{wave}), their derivatives, and the constant density $\bar \rho>0$. The notation $A_1\les B_1$ means $A_1\le C B_1$ with $C>0$ a universal constant; $A_1\ges B_1$ is understood in the similar way; and $A_1\approx B_1$ means both inequalities hold.  We write $A_1=O(B_1)$ means $|A_1|\les |B_1|$, and  $A_1=o(B_1)$ if $|A_1/B_1|\rightarrow 0$ as $B_1\rightarrow 0$. 

\subsection{Lower-bound theorem on lifepspan}
Since the coefficients of \eqref{wave} depend only on $\bp\phi$, we can apply 
the following lower-bound theorem on lifespan, proved  independently in John \cite{FJohn4} and H\"{o}rmander \cite{Hormander_1}.
\begin{theorem}\cite[Page 129, Theorem 6.5.9]{Hormander}\label{known}
  The Cauchy problem (\ref{wave}) with $f,g\in C_0^\infty({\mathbb R}^3)$ has a $C^\infty$ solution $\phi$ for $0\le t<T_\ve$, and satisfies
  \begin{equation*}
\liminf_{\ve\rightarrow 0} \ve \log T_\ve\ge \tau_*
\end{equation*}
where $\tau_*=(\mbox{max}_{\omega, q}\f12 G(\omega) \p^2 F_0(\omega, q)/\p q^2)^{-1}$ with $q=r-\bar c t$ and $\omega\in {\mathbb S}^2 $.
Here $G=-2\bar c^{-1}\wp$, $\ve F_0$ is the Friedlander radiation field of the Cauchy problem of the associated linear wave equations, 
\begin{equation}\label{5.5.2.26}
 F_0(\omega,q)=\frac1{4\pi}\Big(\bar c^{-1}\mathcal Rg(\omega,q)-\partial_q \mathcal Rf(\omega,q)\Big),
 \end{equation}
with $\mathcal R$  the Radon transform
$\mathcal Rh(\omega,q)=\int_{y\cdot\omega=q} h(y)\,dS_y$
and $\omega\in {\mathbb S}^2$.   
\end{theorem}

The calculation for $G=-2\bar c^{-1}\wp$ can be found in Proposition \ref{G-value} in Section \ref{Apdx}. 
We can prove that $\tau_*\ge 0$ and it vanishes only when  $f=0=g$. (See \cite[Theorem 6.2.2 and Lemma 6.5.4]{Hormander}.) Moreover, with the maximum in the definition of $\tau_*$ achieved at $(\omega_*, q_*)$,  the fact that $F_0(\omega, q)$ is supported within $|q|\le 1$  gives the range for $q_*$. The maximizing point identifies the outgoing direction and retarded time along which the lower-bound analysis predicts the first singular behavior. 

We choose  $$t_0=5 \bar c^{-1}$$ so that all outgoing null geodesics initiated from the support of the radiation field have  entered the exterior annular region where the acoustical-geometric estimates will be applied.  
 We will pick the outgoing null geodesic  
 \begin{equation}\label{7.26.1.26}
 \Upsilon_*(t) \mbox{ starting at } t=t_0, \mbox{ from } (\omega_*, q_*),
 \end{equation}
  which plays a crucial role in our proof.

\subsection{Main theorem and strategy of the proof}
In the context of the compressible Euler equations, a shock forms when the solution $(\varrho, v)$ remains bounded and continuous, while its first derivatives become unbounded. Geometrically, this corresponds to the intersection of neighboring characteristic hypersurfaces.

 The goal of this paper is to prove the following main theorem
\begin{theorem}[Main Theorem]\label{main}
There exists a small $\ve_0>0$, depending only on $f, g$ in (\ref{wave}), their derivatives and the constant state $\bar \rho>0$, such that, for any  $0<\ve\le \ve_0$,  
 (\ref{wave}) has a unique smooth solution for $0<t<T_\ve$, where the lifespan of the smooth solution $T_\ve$ satisfies
\begin{equation}\label{5.10.1.26}
    \lim_{\ve\rightarrow 0} \ve \log T_\ve=\tau_*.
\end{equation}
Moreover, the boundary of the maximal Cauchy development of the solution contains a shock at $t=T_\ve^e$ which satisfies  $T_\ve^e\ge T_\ve$, 
\begin{equation}\label{5.13.3.26}
        \lim_{\ve\rightarrow 0} \ve \log T^e_\ve=\tau_*,
\end{equation}
and, with $\Phi=(v, \varrho)$, when approaching the shock
 $$\limsup_{t\rightarrow T_{\ve-}^e} |\bp\Phi|=\infty.$$ 
\end{theorem}
\begin{remark}
Theorem \ref{main} confirms that the solution arisen from sufficiently small compactly supported, smooth perturbation develops singularity of John's type, i.e. satisfying a), b) and c) in Section \ref{mot}. 
\end{remark}
\begin{remark}
We actually prove that the blow-up of $\bp\Phi$ occurs along an outgoing null geodesic approaching the shock.
\end{remark}
\subsubsection{Strategy of the proof} The proof combines John-H\"{o}rmander's radiation-field analysis with the exterior acoustical geometry result in Christodoulou-Miao \cite{Miao_thesis}.   The main new idea is   to provide   the   mechanism   for the evolutionary formation of a compression region, prior to the blow-up,  whereas  all  current shock formation or blow-up  results for Euler equations assume compression or some other restriction for the initial data.\footnote{ Recall that \cite{shock_demetrios}, \cite{Miao_thesis} and \cite{Buck_2}-\cite{Buck_4},  \cite{Spck-luk_2}-\cite{Pin-Shuang}  impose a compression condition on the  initial data, and that  Alinhac imposed in \cite{Alin2} the non-degeneracy condition on the radiation profile of the data.} We prove that the compression forms  at the intermediary time  $t_1=\frac{1}{2\ve}>t_0>0$.  Crucial  for achieving that   is the derivation of    a  Riccati-type equation along $\Upsilon_*$ in (\ref{7.26.1.26}) (as  well as neighboring ones),  for a    specifically designed  quantity $ \ti z$, to which we refer as
the \textit{physical radiation field}.    Inspired by \cite{Wang2024}    we choose   that quantity  to be
\bea
\label{eq:phys.radfield}
  \ti z=c\widetilde r(\Lb\varrho-h\varrho),\qquad
  \tir=t+u.
\eea
Here   $u$ is the acoustical retarded time defined in (\ref{optical}), \,  $\Lb$ is a null vector field transversal to the level set of $u$, defined in (\ref{bb1}), \,  and
$h=\f12\operatorname{tr}\chi$ is half the null expansion of the outgoing
acoustical spheres.

Once those are established we obtain the upper bound for the
exterior lifespan $T_\ve^e$, and identify the first breakdown in the exterior region as shock
formation. We then prove the blow-up of the first derivative of solution $\Phi=(v, \varrho)$  approaching the shock. It turns out that  the upper bound of the exterior lifespan agrees with the lower bound of the lifespan given by Theorem \ref{known} in the asymptotic sense. This completes the proof of the main theorem. We illustrate in Figure \ref{figI} our strategy and provide further details below.


\begin{figure}[htbp]
    \centering

\scalebox{0.75}{
\begin{tikzpicture}[
    xscale=1.5, 
    yscale=1.0,
    dot mark/.style={
        postaction={
            decorate, 
            decoration={
                markings, 
                mark=at position #1 with {
                    \fill (0,0) circle (1.8pt);
                    \coordinate (bluedot) at (0,0);
                }
            }
        }
    }
]

    \def\H{6}               
    \def\B{1}               
    \def\TW{3.5}            
    \def\TWin{2.8}          
    \def\Ystep{0.5}         
    
    \def\XoutStep{1.2083}   
    \def\XinStep{0.75}      
    
    \def\E{0.22}

    \draw[thick, dash dot, -latex] (0, -1) -- (0, \H + 1.2) node[above] {$t$};

    \draw[thick, dashed] (\B, 0) arc (0:180:{\B} and {\B*\E});
    
    \draw[thick] (-\B, 0) -- (-\TW, \H); 
    \draw[thick] (\B, 0) -- (\TW, \H); 

    \draw[thick] (-\XinStep, \Ystep) to[bend right=10] (-\TWin, \H); 
    \draw[thick] (\XinStep, \Ystep) to[bend left=10] (\TWin, \H); 

    \draw[thick, dashed, gray] (0, \H) ellipse ({\TW} and {\TW*\E});
    \draw[thick] (0, \H) ellipse ({\TWin} and {\TWin*\E});

    \fill[mypink, fill opacity=0.3, even odd rule] 
        (0, \Ystep) ellipse ({\XoutStep} and {\XoutStep*\E}) 
        (0, \Ystep) ellipse ({\XinStep} and {\XinStep*\E});
    \draw[line width=1.5pt, mypink] (0, \Ystep) ellipse ({\XoutStep} and {\XoutStep*\E});
    \draw[line width=1.5pt, mypink] (0, \Ystep) ellipse ({\XinStep} and {\XinStep*\E});

    \draw[thick] (-\B, 0) arc (180:360:{\B} and {\B*\E});

    \draw[thick, myblue, dot mark=0.55] 
        (0.85, \Ystep) .. controls (1.2, 2.5) and (2.2, 4.5) .. (3.1, \H + 0.5);
    \fill[myblue] (0.85, \Ystep) circle (1.8pt);

    \node[font=\Large] at (-2.2, 3.5) {$D^+$};

    \node[anchor=north west, myblue, font=\small] at (0.85, \Ystep - 0.05) {$(\omega_*, q_*)$};
    \node[anchor=south west, myblue, font=\large] at (3.1, \H + 0.5) {$\Upsilon_*$};
    
    \node[anchor=west, myblue] at ([xshift=6pt]bluedot) {$t_1 = 1/2\varepsilon$};

    \draw[dashed] (-\XoutStep, \Ystep) -- (-4.5, \Ystep) node[left] {$t_0 = 5\bar{c}^{-1}$};
    \draw[dashed] (-\B, 0) -- (-4.5, 0) node[left] {$t = 0$};

    \draw[<-, >=stealth, line width=1.5pt, mypink] 
        (\XoutStep + 0.1, \Ystep) -- (3.5, \Ystep) 
        node[right, black] {\Large $\Sigma_{t_0}^m$};

\end{tikzpicture}
}
\caption{Set-up}
    \label{figI} 
\end{figure}

 \subsubsection*{Step 1: Geometric set-up.} 
Using the standard local energy estimates, we first propagate energies from $t=0$ to $t_0= 5 \bar c^{-1}$. At $t_0=5\bar c^{-1}$, we consider the annulus $\Sigma_{t_0}^m$, with $q_0=r-\bar c t_0\in [-1, 1]$, where $r=1+\bar c t_0$ is exactly the intersection of $t=t_0$ with the outer Minkowskian boundary. \begin{footnote}{This annulus is sufficiently thin and far away from the central axis. It matches with the setting in \cite{Miao_thesis}.}\end{footnote} We set up the standard radial foliation at $\Sigma_{t_0}^m$, i.e. 
$\bar c u=q_0$ at $t=t_0$, and define the acoustical function $u$ by using Eikonal equation  throughout the outgoing development of $\Sigma_{t_0}^m$, denoted by $D^+$ (see Definition \ref{def-causal}). 
This allows us to have the full set of geometric estimates of \cite[Theorem 17.1]{Miao_thesis} in $D^+$ for $t\in (t_0, T_\ve^e)$, where $T_\ve^e$ is the (exterior) lifespan of the outgoing development of $\Sigma_{t_0}^m$.  We summarize this set of estimates in Proposition \ref{decay_ricci_coef}.

\subsubsection*{ Step 2: Riccati equation for the physical radiation field}

Using the wave equation  \eqref{4.10.2.19} for $\varrho$,
\[
  \Box_{\bg}\varrho=\sQ^0,
\]
together with the causal-geometric estimates in $D^+$   (Proposition  \ref{decay_ricci_coef}) one obtains, along each
outgoing null geodesic $\Upsilon_{\omega,q}$ which initiates at $(\omega, q, t_0)$,  the Riccati-type equation, see \eqref{Riccati_ap},
\bea
\label{eq:MainRicatti-intro}
  \frac{d}{dt}\ti z(t)
  =
  \frac12\wp\,\tir^{-1}\ti z(t)^2
  +\Er_1(t)\ti z(t)
  +\Er_0(t),
  \quad t\in[t_0,T_\ve^e),
\eea
with proper control on the errors $\Er_1$  and $\Er_0$. The constant $\wp=(\gamma+1)/2$, due to the equations of state,  gives the desired positive sign.

Moreover, with the null lapse $\bb$, defined by
$\bb^{-1}=-\bT(u)$ in the outgoing acoustical foliation, we derive in (\ref{5.6.1.26}),
\bea
\label{eq:propag-bbz}
  L(\bb z)=O\!\left(\ve\log\langle t\rangle\langle t\rangle^{-2}\right),
  \qquad z=c^{-1}\widetilde z,
\eea
which shows that $\bb z$ does not vary much along the outgoing null geodesics. This is crucial for  proving the blow-up of $\bp\Phi$ 
at shock formation.

\subsubsection*{Step 3:  The key compression estimate}

In order to show the solution $\ti z$ of (\ref{eq:MainRicatti-intro}) blows up   within finite time, it is crucial to show
\begin{equation}\label{intro_1}
\ti z(\Upsilon_*(t_1))=-2\ve \p_q^2 F_0(\omega_*, q_*)+o(\ve),\,\, \mbox{ with } t_1=\frac{1}{2\ve}.
\end{equation}
Note that the right-hand side is positive, i.e. for sufficiently small $\ve>0$,  compression occurs at $t_1$.

To see (\ref{intro_1}), going back to the equation of (\ref{wave}) of the irrotational compressible Euler flow,  for $0<t< 1/\ve$, we approximate the solution of $\phi$ by the linear wave $\phi_\ve$  defined in (\ref{12.30.2.25}).
In Section \ref{pmain}, we first approximate $\ti z$ at $t_1=\frac{1}{2\ve}$ by $-2\p_q(r\varrho)$ in Lemma \ref{comp_radiation}. 
Using the results in Section \ref{prl_est} to compare the solution at $t=t_1$ with Friedlander's radiation field $\frac{\ve}{r}F_0$, together with the Bernoulli formula \eqref{Bernoulli} and the inverse function theorem, we obtain the key approximation result (see Proposition \ref{7.03.1.26})
\begin{equation}\label{7.22.1.26}
\p_q(r\varrho)(\omega, q)= \ve \bar c^{-1}\p^2_q F_0(\omega, q)+\mbox{ error},\,\,\mbox{ for }(\omega, q)\in D^+\cap \Sigma_{t_1}.
\end{equation}
 In Corollary \ref{4.17.1.26} in Section \ref{decay}, we prove that the variation of the values of $q$ and $\omega$ along an outgoing null geodesic starting at $\Sigma_{t_0}^m$ is negligible. Therefore applying (\ref{7.22.1.26}) to  $\Upsilon_*(t_1)$, the right-hand side (\ref{7.22.1.26}) can be approximated by $\ve \bar c^{-1}\p^2_q F_0(\omega_*, q_*)$ modulo error terms, which gives (\ref{intro_1}).
\subsubsection*{Step 4: Shock formation and lifespan}
 Using  (\ref{intro_1})  as initial data for (\ref{eq:MainRicatti-intro}), we provide a finite upper-bound of the blow-up time for the solution $\ti z(t)$,  and derive
\begin{equation*}
\lim_{\ve\rightarrow 0}\sup \ve \log T_\ve^e\le \tau_*.
\end{equation*}
Using this result and $T_\ve^e\ge T_\ve$, together with Theorem \ref{known}, and the breakdown criterion given by \cite[Theorem 17.1]{Miao_thesis} (see Proposition  \ref{decay_ricci_coef} (2)), we conclude  shock formation  and obtain the lifespan results in Theorem \ref{main}. 

\subsubsection*{Step 5: Blow-up of $\bp\Phi$}  
Shock formation, characterized by the vanishing of $\bb$, is confirmed in Step 4 by the breakdown criterion of \cite[Theorem 17.1]{Miao_thesis} (see Proposition \ref{decay_ricci_coef} (2)), according to which a small-data solution can break down in the exterior region $D^+$ at $T_\ve^e<\infty$ only if $\bb\rightarrow 0$ as $t\rightarrow T_{\ve-}^e$. 
This criterion, however, guarantees only the existence of a shock. It does not determine its precise location in $\Sigma_t^m$ as $t\to T_{\ve-}^e$. Under an additional initial compression assumption, such as that imposed in \cite{Miao_thesis}, one can track the evolution of the initial compression region and thereby locate the shock. 

  In Step 3, we discussed how the compression develops at $t=t_1$ along $\Upsilon_*$. It remains possible, however, that a shock forms earlier along another null geodesic. This precludes us from concluding that $\bb\rightarrow 0$ or $\bp\Phi$ blows up along $\Upsilon_*$ at $T_{\ve-}^e$.
  
  To complete the proof of Theorem \ref{main}, we show by contradiction that $z$ blows up along an outgoing null geodesic approaching the shock formed at $T_\ve^e$, which gives the blow-up of $\bp\Phi$. To see this, we derive in (\ref{7.22.2.26}) the structure equation
$$
L\bb=-\f12\wp \bb z\tir^{-1}+\mbox{ error}
$$
and rely on the property that the variation of $\bb z$ is negligible along an outgoing null geodesic due to (\ref{eq:propag-bbz}) (see (\ref{5.6.1.26})). We then proceed by contradiction as follows. Suppose that along an outgoing null geodesic, say $\Upsilon_1$, $\bb\rightarrow 0$ as $t\rightarrow T_\ve^e$, while $z$ stays bounded. Using the crucial upper-bound estimate (see (\ref{5.12.1.26})) of the blow-up time of $\ti z$ derived in Step 4, we prove that compression must have already formed at $\Upsilon_1(t_1)$, where we have
 $$\bb z\ges \ve, \quad \bb\approx \bar c.$$
  Nevertheless, approaching the shock with $z$ bounded along $\Upsilon_1$ implies $\bb z({T_{\ve-}^e})=0$. This contradicts the fact that the variation of $\bb z$ is negligible along the outgoing null geodesic $\Upsilon_1$, as discussed in Step 2.  

\begin{remark}
As mentioned earlier, the main novelty of our result is that it applies to general, smooth, compactly supported initial data without any restriction beyond smallness, in contrast to existing works. 
 In \cite[Chapter 18]{Miao_thesis},  Christodoulou–Miao  proved 
 $$\bb\approx 1+E(t)(1+\log t),$$  where the initial compression condition forces $E(t)$ to be uniformly negative when $t$ is sufficiently large.
 Thus, the right-hand side vanishes in finite time, and a shock forms. By contrast our  strategy is based on Riccati-type  equations which, together with  our   derivation of a compression  regime at the intermediate time $t_1 =\frac{1}{2\ve}$,   
 drive the exterior solution  to blow up in finite time. Shock formation is then a consequence of the criterion  established in \cite[Theorem 17.1]{Miao_thesis} (see Proposition  \ref{decay_ricci_coef} (2)). However, this argument guarantees only the existence of a shock at $T_{\ve-}^e$ and does not directly yield the blow-up of $\bp\Phi$ at the shock. In previous works, the latter can be obtained using the initial compression assumption and a more direct analysis of the evolution of the solution. Instead, our proof of the blow-up of $\bp\Phi$ at the first shock in the exterior region is based on a contradiction argument, explained in Step 5. 
\end{remark}

\begin{remark}
 In previous works such as \cite{Alinhac} and \cite{Yin}, the construction of the quantity, that verifies the desired Riccati equation, relies crucially on radial symmetry. This reduces the dynamics to an essentially one-dimensional $2\times 2$ system. Even after this substantial simplification, identifying a quantity that obeys the desired Riccati equation requires a delicate analysis. Such a construction would be even more challenging in the general case if we were to follow the same approach. Instead,
  the construction of (\ref{eq:phys.radfield}) follows that of \cite{Wang2024}, where $ \ti z$  plays a crucial role in controlling the dynamics via the Riccati equation.
\end{remark}
\subsection{Organization of the paper}
In Section \ref{prl_est}, we give preliminary results on approximating the solution of (\ref{wave}) in terms of the Friedlander radiation field within a moderately long time.

In Section \ref{causal}, we give the geometric set-up for the outgoing development of $\Sigma_{t_0}^m$, denoted by $D^+$. We define the acoustical function $u$, and derive in Lemma \ref{7.03.2.26} the key Raychaudhuri equation in the acoustical spacetime, whose error terms have good decay properties. 
Then we derive in Proposition \ref{7.03.3.26} the transport equation  of $z$ and $\bb z$, with $z=c^{-1}\ti z$, along the outgoing null geodesics.

In Section \ref{decay}, we recast the results of \cite[Theorem 17.1]{Miao_thesis} in Proposition \ref{decay_ricci_coef} in $D^+$, and use it to derive in Proposition \ref{7.9.4.26} comparison estimates on $\bar c \tir/r$, $\la$ (defined in Section \ref{causal}), etc. Using the comparison estimates, we then control in Corollary \ref{4.17.1.26} the drifting of the values of $q$ and $\omega$ along an outgoing null geodesic.

In Section \ref{pmain}, we prove Theorem \ref{main}. We first prove Proposition \ref{12.31.1.25} which gives (\ref{intro_1}) using the results in Section \ref{prl_est}, the inverse function theorem and Bernoulli law. Using Proposition \ref{7.03.3.26}, we derive the Riccati equation in Proposition \ref{Riccati_prop} for $\ti z$ with proper error control, and control $L(\bb z)$ as in (\ref{eq:propag-bbz}). Then using Proposition \ref{12.31.1.25} and Proposition \ref{Riccati_prop}, we prove Theorem \ref{main}.

In Section \ref{Apdx}, we give the calculation of $G(\omega)$ in Theorem \ref{known}, a set of geometric calculations used for giving the comparison estimates in $D^+$ in Section \ref{decay} and the derivation for (\ref{4.10.1.19}) and (\ref{4.10.2.19}) in Section \ref{Re_wave}.

\section{Preliminary control of the solutions}\label{prl_est}
Consider the linear wave equation 
\begin{equation}\label{lwv}
    (\p_t^2 -{\bar c}^2 \Delta) W(t,x)=0, W(0,x)=f, \p_t W(0, x)=g.
\end{equation}
 For $0\le t<1/\ve$, 
we construct the approximate solution
\begin{equation}\label{12.30.2.25}
\phi_\ve(t,x)=\ve W(t,x).
\end{equation}
We recall the following results from \cite{FJohn4} and \cite[Page 124-126]{Hormander}.
\begin{proposition}\label{12.31.4.25} With $\ve>0$ sufficiently small, 
if $0<t<1/\ve$,  $\phi_\ve$ defined in (\ref{12.30.2.25}) is $C^\infty$,   
and it satisfies
\begin{align*}
|Z^I \phi_\ve(t,x)|&\le C_{I}\ve(1+t)^{-1}\\
    |Z^I\bp(\phi-\phi_\ve)|&\le C_{I}\ve^2\log(1/\ve)(1+t)^{-1},
\end{align*}
where $Z$ denotes the commuting vector fields $\p_i, \p_t$, $\bar c t\p_i +\bar c^{-1}x^i\p_t$, $\O=\tensor{\ud\ep}{^a_{ij}} x^i \p_j$, and the scaling vector field $ t\p_t+r\p_r$.      
Here $|I|\ge 0$, and $C_{I}$ are independent of $\ve$.

\end{proposition}
 In what follows we introduce the notation $f=O^{\le n} (\mathfrak{N})$, $O_q^{\le n} (\mathfrak{N})$ to denote error terms of size $\mathfrak{N}$. More precisely,     
\begin{equation*}
f=O^{\le n} (\mathfrak{N}), \mbox{ if }|f, \p^\a f|\les \mathfrak{N},\, |\a|=1, \cdots, n
\end{equation*}
and
\begin{equation*}
f=O_q^{\le n} (\mathfrak{N}), \mbox{ if }|f, \p_q^l f|\les \mathfrak{N},\, l=1, \cdots, n.  
\end{equation*}

We will relate the linear wave to the Friedlander radiation field in Proposition \ref{apx_prop}, which is a classical result. For completeness, we include the proof below.

\begin{lemma}\label{3.31.3.26}
For  $h\in C_c^\infty(\mathbb R^3),$ consider the spherical mean 
$$I_h(r,\omega,q):=\int_{\mathbb S^2} h(r\omega+\tau\theta)\,d\theta,
\qquad \tau=r-q.$$
Then, for $r\gg 1$, 
$$I_h(r,\omega,q)=\frac1{(r-q)^2}\,\mathcal Rh(\omega,q)+O^{\le 3}(r^{-3})$$
holds
uniformly for $q$ in compact sets, where $\R h$ denotes the Radon transform of $h$, see (\ref{5.5.2.26}).
\end{lemma}
\begin{proof}

By rotation invariance, we may assume $\omega=e_3$. Then the sphere
$|y-re_3|=\tau,\, \tau=r-q,$
can be written as a graph over the plane $y_3=q$ (taking the branch near $y_3=q$):
$$y_3=q+\sY(r,\omega,q;y'),
\qquad y'=(y_1,y_2)\in\mathbb R^2,$$
where $\sY=\tau -\sqrt{\tau^2 - |y'|^2}$
is smooth and, for $q$ in a fixed compact set,
$$\sY=O(r^{-1}),\qquad
\partial_{r,\omega,q}^\beta\sY=O(r^{-1}),\quad |\beta|\le 4.$$
Similarly, the surface measure has the form
$$dS = \sqrt{1+|\nab_{y'}\sY|^2} dy'=J(r,\omega,q;y')\,dy',$$
with
$$J=1+O(r^{-1}),\qquad \partial_{r,\omega,q}^\beta J=O(r^{-1}),\quad 1\le |\beta|\le 3.
$$
Since $h$ is compactly supported, the $y'$-integration is over a fixed compact set. Hence by coarea formula, we derive
$$\tau^2 I_h(r,\omega,q)
=
\int_{\mathbb R^2} h\big(y',\,q+\sY(r,\omega,q;y')\big)\,J(r,\omega,q;y')\,dy'.$$
According to (\ref{5.5.2.26}), 
$\mathcal Rh(\omega,q)=\int_{\mathbb R^2} h(y',q)\,dy'$, it follows that 
$$\tau^2 I_h-\mathcal Rh
=
\int \big[h(y',q+\sY)-h(y',q)\big]J\,dy'
+\int h(y',q)\,(J-1)\,dy'.$$

For the first term, Taylor’s theorem in the $y_3$-variable gives
$$h(y',q+\sY)-h(y',q)
=
\sY\partial_{y_3}h(y',q)
+
O(\sY^2).$$
Since $\sY=O(r^{-1})$, this is $O(r^{-1})$. Differentiating in $(r,\omega,q)$ up to order $3$ does not change the order of the errors, because every derivative of $\sY$ and $J$ is still $O(r^{-1})$, and $h$ is smooth and compactly supported.

For the second term, $J-1=O(r^{-1})$, and again the same remains true after up to three derivatives in $(r,\omega,q)$.
Therefore
$$
\partial_{r,\omega,q}^\beta\big(\tau^2 I_h-\mathcal Rh\big)=O(r^{-1}),
\quad |\beta|\le 3.
$$
Now we define
$$B_h(r,\omega,q):=r\big(\tau^2 I_h(r,\omega,q)-\mathcal Rh(\omega,q)\big),$$
then $B_h$ and its first three derivatives in $(r,\omega,q)$ are bounded:
$$
\partial_{r,\omega,q}^\beta B_h=O(1),\quad |\beta|\le 3.
$$
We now control the spacetime derivatives of the quantity
$$E_h=I_h-\tau^{-2}\mathcal Rh
=
\tau^{-2}r^{-1}B_h(r,\omega,q).$$
A spacetime derivative $\partial^\alpha$ of order $|\alpha|\le 3$ hits either
		the factor $\tau^{-2}r^{-1}$, or
	$B_h(r,\omega,q)$.
Now,
	$\tau\sim r$, so $\tau^{-2}r^{-1}=O(r^{-3})$;
each derivative of $\tau^{-2}r^{-1}$ is at worst $O(r^{-3})$ or better;
		derivatives of $B_h(r,\omega,q)$ are bounded. 
Hence
$$
\partial^\alpha E_h=O(r^{-3}),\quad |\alpha|\le 3.
$$
Therefore
$$
\partial^\alpha I_h
=
\partial^\alpha(\tau^{-2}\mathcal Rh)+O(r^{-3}),
\quad |\alpha|\le 3
$$
as stated.
\end{proof}
\begin{proposition}\label{apx_prop}
For $\f12 t_1<t< 2 t_1$ and with $|q|\les 1$,
\begin{equation*}
\phi_\ve(t,x)=\ve r^{-1} F_0(\omega, q)+\R_\ve(t,x), 
\end{equation*}
with $F_0(\omega, q)$ defined in (\ref{5.5.2.26}) and  the following error estimates  
\begin{equation*}
|\mathcal R_\varepsilon| + |\p^\a \mathcal R_\varepsilon|
\les \varepsilon r^{-2}, |\a|=1,2.
\end{equation*}
\end{proposition}
\begin{proof}
   
For $x=r\omega$ and $q=r-\bar c t$ fixed, we claim
\begin{equation}\label{4.2.1.26}
W(t,r\omega)
=
\frac{1}{r}\,F_0(\omega,q)+O^{\le 2}(r^{-2}),
\end{equation}
with $W$ defined by (\ref{lwv}) and $F_0(\omega, q)$ given in (\ref{5.5.2.26}). The estimate of $\phi_\ve(t,x)$ is then a consequence of the definition (\ref{12.30.2.25}).

Indeed the Kirchhoff formula expresses $W$ in terms of the spherical means of $f$ and $g$ as follows:
$$W(t,x)
=\frac{1}{\bar c}
\partial_t\!\left(\frac{\bar c t}{4\pi}\int_{\mathbb S^2} f(x+\bar c t\theta)\,d\theta\right)
+
\frac{t}{4\pi}\int_{\mathbb S^2} g(x+\bar c t\theta)\,d\theta.$$

Since $r=\bar c t+q$ and $r\approx t$, we have
$\frac1{\bar c^2 t}=\frac1{\bar c r}+O(r^{-2})$.
Using Lemma \ref{3.31.3.26}, we deduce, 
$$
\frac{t}{4\pi}\int_{\mathbb S^2} g(r\omega+\bar c t\theta)\,d\theta
=
\frac{t}{4\pi}\left((\bar c t)^{-2}\mathcal Rg(\omega,q)+O^{\le 3}(t^{-3})\right)
=
\frac1{4\bar c\pi r}\mathcal Rg(\omega,q)+O(r^{-2}).
$$
 Since $q=r-\bar c t$, we have $\partial_t q=-\bar c $, 
and thus
\begin{align*}
\frac{1}{4\pi}\p_t\left(t \int_{\mathbb S^2} f(x+\bar c t \theta) d\theta\right)&=\frac{1}{4\pi}\int_{\mathbb S^2}f(x+\bar c t\theta) d\theta-\frac{\bar ct}{4\pi}\p_q \left(\int_{\mathbb S^2}f(r\omega+(r-q)\theta) d\theta\right).
\end{align*}
Using Lemma \ref{3.31.3.26},  it follows that 
$$\frac{1}{4\pi}\int_{\mathbb S^2} f(r\omega+\bar c t\theta)\,d\theta
=
\frac1{4\pi (r-q)^2}\mathcal Rf(\omega,q)+O^{\le 3}(r^{-3});$$
and 
\begin{align}
\frac{\bar ct}{4\pi}&\p_q \left(\int_{\mathbb S^2}f(r\omega+(r-q)\theta) d\theta\right)\nn\\
&=\frac{\bar c t}{4\pi}\left(\p_q(\frac{1}{(r-q)^2}\R f(\omega, q))+O^{\le 2}(r^{-3})\right)\nn\\
&=\frac{\bar c t}{4\pi(r-q)^2}\p_q \R f(\omega, q)+\frac{1}{2\pi(r-q)^2}\R f(\omega, q)+O(r^{-2})\nn\\
&=\frac{1}{4\pi}(r^{-1}+O(r^{-2})) \p_q \R f(\omega, q)+\frac{1}{2\pi(r-q)^2} \R f(\omega, q)+O(r^{-2}).\nn
    \end{align}

Putting them together,
$$W(t,r\omega)
=
\frac1{4\pi r}\Big(\bar c^{-1}\mathcal Rg(\omega,q)-\partial_q\mathcal Rf(\omega,q)\Big)
+
O(r^{-2}).$$ 
Thus we have  
$$W(t,r\omega)=\frac1r\,F_0(\omega,q)+O(r^{-2}).$$

The same argument applies after differentiating the Kirchhoff formula, since differentiating f or g under the integral gives the same type of sphere integral. Repeating the above calculation using Lemma \ref{3.31.3.26} when it is differentiated by $\p^\a$ for $|\alpha|\le2$,  we have
$$\partial^\alpha W
=
\partial^\alpha\!\big(r^{-1}F_0(\omega,q)\big)
+
O(r^{-2}),$$
as stated in (\ref{4.2.1.26}).

\end{proof}

\section{Causal geometry in the acoustical spacetime}\label{causal}
Proposition \ref{12.31.4.25} only provides us with the control for $0<t<1/\ve$. To prove Theorem \ref{main}, however, we need control the solution in $D^+$ until $T_{\ve-}^e$, when singularity forms. To obtain such control, we will rely on \cite[Theorem 17.1]{Miao_thesis}. Since both the theorem and our blow-up mechanism are formulated in terms of the intrinsic acoustical frame, we begin by introducing the corresponding geometric setup, compatible with \cite{Miao_thesis}, and derive the crucial geometric structural equations.

We define an acoustical function $u$ to be the solution of the Eikonal equation
\begin{equation}\label{optical}
\bg^{\a\b}\p_\a u \p_\b u=0,
\end{equation}
where $\bg$ is the acoustical metric (\ref{metric}), 
and set 
$$u= \bar c^{-1} r-t_0, \bar c u\in[-1,1] \mbox{ at } t=t_0.$$

Note that the result of \cite{Miao_thesis} requires the proper setting at the slice of $\{t=t_0\}$. In particular, the analysis is on the exterior solution for data prescribed in a thin annulus near the boundary of the support. To match with its setting,  we  will take  $t_0=5 \bar c^{-1}$ as the initial slice to apply this result. 

\begin{definition}\label{def-causal}
We define
 $$T^e_\ve:=\sup\{t: \mbox{the exterior solution with } u\in [-\bar c^{-1}, \bar c^{-1}] \mbox{ is smooth in } [t_0, t)\}.$$
  Clearly $T_\ve^e\ge T_\ve$. 
  We fix the convention, with $r_M=\bar c t+1$, that 
 \begin{align*}
D&=\{ r\le r_M, 0\le t< T_\ve\},\quad D_0=\{r\le r_M, 0\le t\le t_0\},\\
  D^+&=\{-\bar c^{-1}\le u\le \bar c^{-1},\,  t_0\le t<T^e_\ve\},\quad \Sigma_t^m=D^+\cap\{t'=t\}.
 \end{align*} 
We denote the level sets of the acoustical function $u$ by $\H_u$ as the acoustical null cones in the exterior region $D^+$. 
\end{definition}

 
At $t=t_0$ with $u\in[-\bar c^{-1}, \bar c^{-1}]$,  for each  $\omega\in {\mathbb
S}^2$,  define the null vector field $\hat L$  to be the generator of  the null geodesic $\Upsilon_\omega$ in the acoustical spacetime by
\begin{equation}\label{6.29.2.19}
\bd_{\hat L} {\hat L}=0, \quad \frac{d}{ds}\Upsilon_\omega(s)=\hat L,\quad  \hat L(s)=1, 
\end{equation}
where $s=0$ as $\Upsilon_\omega$ initiates at $\Sigma_{t_0}^m$.  
Let
\begin{equation}\label{bb1}
\bb^{-1}:=-\l \hat L,\bT\r=-\bT(u),\quad e_4=\bb \hat L=\bT+\bN, \quad e_3:=\bT-\bN=2\bT-e_4, 
\end{equation}
\begin{footnote}{ Alternatively we may write $e_4=L$ and $e_3=\Lb$.}\end{footnote} where $\bN=\frac{\nab u}{|\nab u|}$ with $\nab=\nab_e$ the Levi-Civita connection of the induced metric, which is exactly the Euclidean metric, on $\Sigma_t$. Clearly
 \begin{equation*}
 \bT u=-\bb^{-1}=-\bN u.
 \end{equation*}
   
 We define in $D^+$
$$\tir=t+u, \, y^k= \bN^k-\frac{x^k}{c\tir},\, {y'}^k=\bN^k-\frac{x^k}{r}.$$

Let ${}\rp{a}\O=\tensor{\ud\ep}{^a_j_k}x^j\p_k, a=1,2,3$ and $\Pi^i_j=\delta_j^i-\bN^i\bN_j$. Define the rotation vector-field  by \begin{footnote}{${}\rp{a}\Omega$ can be written as $\Omega$ for short.}\end{footnote}
\begin{equation}\label{5.14.2.23}
{}\rp{a}\Omega^m={}\rp{a}\O^k\Pi_k^m.
\end{equation}
 
  We denote $\la^a=\l {}\rp{a}\O, \bN\r$. This implies
\begin{equation}\label{3.19.1.21}
{}\rp{a}\Omega+\la^a\bN={}\rp{a}\O.
\end{equation}

Due to the definition of $u$ at $t=t_0$ 
$$ \bN=\p_r, \bb=\bar c, \la=0, \mbox{ at }t=t_0 \mbox{ and } \forall r\ge r_M=\bar c t+1.$$

We define $S_{t,u}=\Sigma_t\cap \H_u$. In $D^+$,
for each $u$ the geodesic generators (\ref{6.29.2.19}) define a smooth one-to-one mapping from $S_{t_0,u}$ to $S_{t,u}$. Moreover, since $S_{t_0,u}=\{|x|=\bar c (u+t_0)\}$, we can assign to any point $p$ in $S_{t,u}$ pullback coordinates $\omega=(\omega_1,\omega_2)\in \mathbb S^2$,  by following the null geodesic $\Upsilon_{\omega, q}$ with $q=\bar c u$, from a point  $p_0\in S_{t_0,u}$ 
until it intersects the level set of $t$ at $p$.
$\omega$ and $u$ are constant along the integral curves of $cL$, \begin{footnote}{For fixed $\omega_0\in {\mathbb S}^2$ and $q_0\in[-1,1)$,  values of  $\omega$ and $q$, when expressed in terms of Cartesian coordinates  along $\Upsilon_{\omega_0,q_0}(t)$, vary with $t$. To control these variations, we derive the error estimates in Corollary \ref{4.17.1.26}.}\end{footnote} along which
\begin{equation}\label{11.27.1.23}
c L=\frac{\p}{\p t}.
\end{equation}
We may refer to $p$ by its transported pullback coordinates $p=(t,u, \omega_1, \omega_2)$. 
Let $\ga$ be the induced metric on $S_{t,u}$. We denote by $\sn$ its Levi-Civita connection, and the area  $|S_{t,u}|_\ga=\int_{S_{t,u}} 1 d\mu_\ga$.

Recall from (\ref{metric}) that the components of $\bg$ and $\bg^{-1}$ are
\begin{align*}
&\bg_{00}=-c^2+|v|^2, \quad \bg_{0i}=-v_i, \quad \bg_{ij}=\delta_{ij}\\
&\bg^{00}=-c^{-2}, \quad \bg^{0i}=-c^{-2} v^i \quad \bg^{ij}= \delta^{ij}-c^{-2}v^i v^j.
\end{align*}

Similar to \cite[Section 2]{Wang2024}, we have
 \begin{lemma}\label{dg}
The following formulas hold for a linear differential operator $Y$,  
\begin{equation}\label{10.6.1.22}
\begin{split}
Y(\bg_{\mu\nu}) L^\mu L^\nu&=-2c^{-1}\bN^i Y v_i-2 Y \log c, \quad Y (\bg_{\mu\nu}) \Lb^\mu \Lb^\nu=2c^{-1} Y v^i \bN^i-2Y\log c\\
Y (\bg_{\mu\nu})L^\mu \Lb^\nu&=-6 Y\log c\\
Y (\bg_{\mu\nu})\Lb^\mu e_A^\nu&, Y (\bg_{\mu\nu})L^\mu e_A^\nu=-c^{-1}Y v^i e_A^i,\quad Y(\bg_{\mu\nu}) e_B^\mu e_A^\nu=0.
\end{split}
\end{equation}
Using (\ref{10.6.1.22}), we have
\begin{equation}\label{4.9.1.26}
\Ga_{L\bN}^\bN= L\log c,\, \Ga_{L\bN}^A=-\sn_A\log c.
\end{equation}
where $\Ga_{XY}^Z=\Ga_{\a\b}^\mu X^\a Y^\b Z_\mu$, with $X, Y, Z$ vector fields,  and $\Ga_{\a\b}^\mu$ denotes the Christoffel symbol of $\bg$ under coordinates.
\end{lemma}

Define the second fundamental form
 \begin{equation*}
 k_{ij}=-\f12 \Lie_\bT \delta_{ij}, \qquad \Tr k =\delta^{ij} k_{ij},\qquad \hk_{ij}=k_{ij}-\frac{1}{3} \Tr k \delta_{ij}
 \end{equation*}
 where $\Lie_X$ denotes the Lie derivative by the vector field $X$.

We start with recalling the basic geometric set-up defined by using the null tetrad $\{L, \Lb, e_A, A=1,2\}$, which appeared in \cite{CK}.
Here $\{e_A, e_B\}$ with $A,B=1,2$ is the orthonormal basis of the tangent bundle on $S_{t,u}$.
 The null second fundamental forms $\chi$ and $\chib$,
the torsion $\zeta$, and the Ricci coefficient $\zb$ of the foliation $S_{t,u}$ are defined by
\begin{equation}\label{ricc_def}
\begin{split}
\chi_{AB}=\bg(\bd_A e_4, e_B), &\qquad \chib_{AB}=\bg (\bd_A e_3, e_B),\\
\zeta_A=\f12 \bg(\bd_3 e_4, e_A), &\qquad \zb_A=\f12 \bg (\bd_4 e_3, e_A).
\end{split}
\end{equation}

We also record the following identities from \cite{CK, KR1, Wang2024} under the null tetrad of time foliation
\begin{eqnarray}\label{5.7.1.26}
\begin{split}
\chib_{AB}=-\chi_{AB}-2 k_{AB}, && \zb_A=-k_{A\bN}+\sn_A \log  c\\
\xi_A=k_{A\bN}-\zeta_A+\sn_A\log c, &&\zeta_A=\sn_A\log\bb+k_{A\bN}.
\end{split}
\end{eqnarray}

 We denote by $\tr\chi$ and $\chih$ the trace and traceless part of $\chi$ taken by the metric $\ga$, and apply the same convention to $\chib$.
The second fundamental form of $S_{t,u}$ in $\Sigma_t^m$  is given by
\begin{equation}\label{7.15.7.19}
\theta(X, Y)=\l \nab_X \bN, Y\r
\end{equation}
for any vector fields $X, Y$ tangent to $S_{t,u}$. The trace of $\theta$ is defined by $\tr\theta=\ga^{AB} \theta_{AB}$, and the traceless part of $\theta$ is denoted by $\hat \theta$.

We adopt the metric induced by the null geodesic flow  for the spacetime in $D^+$
\begin{equation*}
-2\bb d u dt+\bb^2 du^2+\ga_{AB}(d \omega^A+\b^A du)(d \omega^B+\b^B du).
\end{equation*}
Note that $\bb \bN=\p_u-\b_A \frac{\p}{\p\omega^A}$, the shift $\b(t_0)=0$ at $\Sigma_{t_0}^m$ and 
 \begin{footnote}
{The formula can be proved directly by using (\ref{lb}) and (\ref{3.19.2}), which will be given shortly.} 
\end{footnote}
\begin{equation}\label{1.22.4.22}
[L, \b]=\bb(\zb-\zeta).
\end{equation}

 \begin{lemma}\label{6.7con}
  The following frame equations hold
\begin{eqnarray*}
\bd_4 e_4=-\bar k_{\bN\bN} e_4,&& \bd_A e_4 =\chi_{AB}e_B-k_{A\bN}
e_4\\
\bd_A e_3=\chib_{AB} e_B+k_{A\bN} e_3,&& \bd_4 e_3=2\zb_A
e_A+\bar k_{\bN\bN}e_3\\
\bd_3 e_4=2\zeta_A e_A+ \ud k_{\bN\bN}e_4, &&\bd_3 e_3=2\xi_A
e_A-\ud k_{\bN\bN} e_3\\
\bd_4 e_A=\sn_4 e_A+\zb_A e_4,&& \bd_3 e_A=\sn_3e_A+\zeta_A
e_3+\xi_A e_4\\
\bd_B e_A=\sn_B e_A+\f12 \chi_{AB} e_3+\f12 \chib_{AB} e_4&& 
\end{eqnarray*}
where $\bar k_{\bN X}=k_{\bN X}-X\log c$ and $\ud k_{\bN X}=k_{\bN X}+X \log c$ for $X\in \T \Sigma$.
\end{lemma}
Using Lemma \ref{dg}, we have
\begin{lemma}[Crucial decomposition for second fundamental forms]\cite[Section 2]{Wang2024}\label{dcom_s}
Let  $\Xi_\mu=\Ga^\eta_{\a\b}(\bg) \bg^{\a\b}\bg_{\eta \mu}$, where $\Ga(\bg)$ is the Christoffel symbol of $\bg$.
Let $\Xi$ be the 1-form
\begin{equation}\label{ricc6.7.2}
\Xi_\ga=(\Ga_{\a\b}^\eta-{\hat \Ga}_{\a\b}^\eta)\bg^{\a\b}\bg_{\ga\eta},
\end{equation}
with $\hat\Ga $ being the Christoffel symbol of a smooth reference metric  $\hat\bg$, which is chosen to be the Minkowski metric $\bm$.
Relative to this frame, let $\eta_{AB}=k_{AB}$ and $\ep_A=k_{A\bN}$ and
 we decompose $\eta$ as
\begin{equation*}
\tr \eta=\delta^{CD}\eta_{CD}, \quad \eta_{AB}=\eh_{AB}+\f12 \delta_{AB}\tr \eta,
\end{equation*}
where $A,B,C,D=1,2$.

Then
\begin{align}
 c k_{ij}&=-\p_i v_j, \quad c\Tr k=-\div v\label{k1}\\
\Xi_\mu\bB^\mu&=\bB\log c-\div v, \, \Xi_j=-\p_j(\log c+\varrho)\label{6.14.1.19}\\
2\bT\varrho&=\wp^{-1}\{\Xi_L+L (\log c+\varrho)\},\quad \Xi_L=\wp \Lb \varrho \label{7.04.7.19} \\
\bar k_{\bN\bN}&=\f12 (\Xi_L-L(\log c+\varrho)-2c^{-1}L v_\bN)\label{7.04.8.19}\\
&=-\wp\bN \varrho-c^{-1}L v_\bN\label{3.22.1.21}\\
\displaybreak[0]
\bN v^i \bN^i&= -\bB\varrho-e_A(v^i) e_A^i, \quad  c^{-1}L v^i \bN^i + L\varrho=\tr\eta \label{7.04.9.19}\\
\displaybreak[0]
\f12(\tr\chi+\tr\chib)&=-(c^{-1}L v_\bN+L\varrho)\label{1.6.1.21}\\
c^{-1}\Lb v_\bN&=\Lb \varrho-(L\varrho+c^{-1}L v_\bN)\label{8.22.3.25}\\ 
c^{-1}Xv^i\Pi_{ij}e_A^j&=\pm k_{A\bN}- \sn \varrho=[\sn\Phi], \,X=L, \Lb, \begin{footnote}{ This identity will be frequently used without mentioning.} \end{footnote}  \label{6.18.1.26}
\end{align}
where $[Y\Phi]=Y\varrho, c^{-1}(Y v)_\bN$, with $Y=L, \Lb, \sn$.
\end{lemma}
\begin{proof}
 For the derivation of 
(\ref{k1})-(\ref{7.04.9.19}), one may refer to \cite[Section 2]{Wang2024}, noting that the metric used there is $c^{-2}\bg$. 
(\ref{1.6.1.21}) follows by using (\ref{7.04.9.19}) and $\f12(\tr\chi+\tr\chib)=-\tr\eta$. 

Due to (\ref{7.04.9.19}) and (\ref{4.23.1.19}), 
\begin{equation*}
c^{-1}\bT v_{\bN}=-\bN \varrho; c^{-1}\bN v_{\bN}=-\bT\varrho+\tr\eta=-\bT \varrho+L\varrho+c^{-1}L v_\bN.
\end{equation*}
This gives (\ref{8.22.3.25}). 

 (\ref{6.18.1.26}) is derived by using (\ref{4.23.1.19}) and the fact that the Euler flow is irrotational. 
 \end{proof}
Using Lemma \ref{6.7con}, we have the following formula
\begin{lemma}
Let $h=\f12 \tr\chi$ and $\hb=\f12 \tr\chib$. For a scalar function $f$, there holds
\begin{align}
\Box_\bg f&=\sD f-L \Lb f-(h-\bar k_{\bN\bN})\Lb f-\hb L f+2\zb^A \sn_A f.\label{6.30.2.19}
\end{align}
where $\sD$ is the Laplace-Beltrami operator of the induced metric $\ga$ on $S_{t,u}$.
\end{lemma}

\begin{proposition}\label{ric44}
 The following important decomposition for $\bR_{44}$ holds true,
\begin{equation}\label{3.7.6.21}
\bR_{44}=L(\Xi_4)-\bar k_{\bN\bN}\Xi_4+ \N(\Phi, \bp \Phi)
\end{equation}
where $\Xi_4=\Xi_\mu e_4^\mu$ with $\Xi_\mu$ defined in (\ref{ricc6.7.2}), and $\N(\Phi, \bp \Phi)$ symbolically represents a set of scalars, taking the following form 
\begin{equation}\label{10.8.1.22}
\N(\Phi, \bp\Phi)=\sum_{\vs(Y_1 Y_2)=0,2} [Y_1\Phi][Y_2\Phi]+|\eh|^2,
\end{equation}
where $Y_1, Y_2$ are in the null tetrad $\{e_4, e_3, e_A, A=1,2\}$, $[Y\Phi]=Y\varrho, c^{-1}(Y v)_\bN$,  and
$\vs(Y)$ is the signature of $Y$ for which we have basic calculation rule:
$$\vs(Y_1 Y_2)=\vs(Y_1)+\vs(Y_2), \, \vs(\Lb):=-1, \vs(\sn):=0, \vs(L):=1.$$ 
\end{proposition}
\begin{remark}
The decomposition of $\bR_{44}$ of the type in (\ref{3.7.6.21}) was first given in \cite{K-commu} and in \cite[Chapter 4]{shock_demetrios}  in the context of fluids. One can refer to  \cite[Section 2]{Wang2024} for the same result of the metric $c^{-2}\bg$.  Here, as in \cite[Section 2]{Wang2024}, it is crucial to verify the precise form of the lower-order quadratic terms on the right-hand side. 
\end{remark}
\begin{remark}
In this paper, we do not rely on the precise formula of $\N(\Phi, \bp\Phi)$ in (\ref{10.8.1.22}). In view of (\ref{8.22.3.25}) and (\ref{6.18.1.26}), we only need the more general symbolic formula, which does not have to be scalar either, 
\begin{equation}\label{6.26.1.26}
\N(\Phi, \bp\Phi)=\Lb \varrho \c \bar \bp\Phi+\bar \bp \Phi\c \bar\bp \Phi
\end{equation}
where $\bar \bp\Phi=c^m(\sn\Phi, L\Phi)$ with $m\in\mathbb Z$, and the powers of $c$ multiplying the components of $\bar \bp \Phi$ may differ. 
\end{remark}
 \begin{proof}
We first recall that relative to any coordinate, the Ricci curvature of $\bg$ can be decomposed as
\begin{equation}\label{ricc6.7.1}
\bR_{\a\b}=-\f12 \Box_\bg (\bg_{\a\b})+\f12 (\bd_\a \Xi_\b+\bd_\b \Xi_\a)+S_{\a\b},
\end{equation}
where $\Xi$ is the 1-form defined in (\ref{ricc6.7.2}).
The term
$S_{\a\b}$ is quadratic in $\bp \bg$, (see \cite[Page 25]{Andersson_Moncrief})
\begin{equation}\label{s44}
\begin{split}
S_{\mu\nu}&=\f12 \bg^{\b\a}\bg^{\sigma \ga}(\p_\nu \bg_{\a \ga}\p_\sigma \bg_{\mu \b}+\p_\mu \bg_{\b \sigma} \p_\ga \bg_{\nu\a}
-\f12 \p_\nu \bg_{\a\ga}\p_\mu \bg_{\b\sigma}\\
&+\p_\b \bg_{\mu\sigma}\p_\a\bg_{\nu \ga}-\p_\b \bg_{\mu\sigma} \p_\ga\bg_{\nu\a})-\f12 \Xi^\a\p_\a \bg_{\mu\nu}.
\end{split}
\end{equation}

Similar to \cite[Proposition 7.5]{rough_fluid}, we derive
\begin{align*}
c^2\Box_\bg \bg_{\a\b}e_4^\a e_4^\b&=\Box_\bg \bg_{\a\b}(\bB+c\bN)^\a(\bB+c\bN)^\b\\
&=\Box_\bg \bg_{\a\b}\bB^\a \bB^\b+2c\Box_\bg \bg_{i\b}\bN^i\bB^\b+c^2\Box_\bg \bg_{i j}\bN^i \bN^j \\
&=\Box_\bg \bg_{00}+2\Box_\bg \bg_{0i}v^i+2c\Box_\bg \bg_{i0}\bN^i\\
&=\Box_\bg (-c^2+|v|^2)-2\Box_\bg v^i(v^i+c\bN^i)\\
&=-2c\Box_\bg c-2\bd_\a c\bd^\a c+2\bd^\a v^i \bd_\a v^i-2c\Box_\bg v_i \bN^i.
\end{align*}
Note 
\begin{equation*}
c^{-1}\Box_\bg c=c^{-1}\bd^\a(c\bd_\a\log c)=\Box_\bg \log c+\bd^\a\log c\bd_\a \log c=(\wp-1)\Box_\bg \varrho+\bd^\a\log c\bd_\a \log c.
\end{equation*}
Combining the above two calculations gives
\begin{align*}
\Box_\bg \bg_{\a\b} e_4^\a e_4^\b=-2(\wp-1)\sQ^0-2c^{-1}\sQ(\bN)-4\sC(\log c, \log c)+2c^{-2}\sC(v^i,v^i)
\end{align*}
where the bilinear form $\sC(f_1, f_2):=\bd^\a f_1 \bd_\a f_2$.

Substituting the above formula to (\ref{ricc6.7.1}), we derive in view of (\ref{4.10.1.19}), (\ref{4.10.2.19}) and Lemma \ref{6.7con} that
\begin{equation}\label{3.7.7.21}
\bR_{44}-(L(\Xi_4)+\bar k_{\bN\bN} \Xi_4+S_{44})=c^{-2}(\sG(v^i, v^j)\delta_{ij}+c\sG(\varrho, v^i) \bN^i)+\sG(\varrho, \varrho),
\end{equation}
where we neglected the particular values of constant coefficients of the right-hand side.

Similar to \cite[Proposition 2.9]{Wang2024}, we use Lemma \ref{dg} to treat $S_{44}$. Also using  (\ref{6.18.1.26}), the first line in (\ref{s44}) contributes the terms $$\sum_{\vs(Y_1 Y_2)=0,2} [Y_1 \Phi][Y_2\Phi].$$

 Note that applying the first line in (\ref{10.6.1.22}) to $Y=\Lb$, also using (\ref{3.22.1.21}) and (\ref{7.04.9.19}), we have
\begin{equation*}
e_4^\mu e_4^\nu \Lb \bg_{\mu\nu}=-4\bar k_{\bN\bN}-2(c^{-2}\bN^j e_4 v^i +e_4 \log c).
\end{equation*}
With the help of the above identity, we derive
\begin{align*}
\begin{split}
-\frac{1}{2}e_4^\mu e_4^\nu \bg^{\la s} \bg^{\rho s'}\p_\rho \bg_{s\mu}\p_\la \bg_{s' \nu}
&=-\f12 e_4^\mu e_4^\nu  \bg^{34}\bg^{34} e_3 \bg_{\a\mu} e_3 \bg_{\b\nu} e_4^\a e_4^\b+\N(\Phi, \bp \Phi)\\
&=-\frac{1}{8} (e_4^\nu e_4^s e_3 \bg_{s\nu})^2+\N(\Phi, \bp \Phi)\\
&=-2\bar k_{\bN\bN}^2+\N(\Phi, \bp \Phi)
 \end{split}
\end{align*}
and we have,  also using (\ref{6.14.1.19}) and (\ref{7.04.8.19}), that
\begin{align*}
-\f12\Xi^\a \p_\a \bg_{\mu\nu}e_4^\mu e_4^\nu=\f12 \bar k_{\bN\bN} e_3 \bg_{44}+\N(\Phi, \bp\Phi)=-2\bar k_{\bN\bN}^2+\N(\Phi, \bp\Phi).
\end{align*}
Moreover, in view of (\ref{7.04.9.19}) and (\ref{6.18.1.26})
\begin{equation*}
  e_4^\mu e_4^\nu\bg^{\b\a}\bg^{\sigma\ga}\p_\b \bg_{\mu\sigma}\p_\a\bg_{\nu\ga} = \sum_{\vs(Y_1 Y_2)=0,2} [Y_1 \Phi][Y_2\Phi]+c^{-2}\sn_A v_B\c\sn_A v_B=\N(\Phi, \bp\Phi).
\end{equation*}
Hence, we conclude
\begin{equation*}
S_{44}=-4(\bar k_{\bN\bN})^2+\N(\Phi, \bp\Phi).
\end{equation*}
(\ref{3.7.6.21}) follows by combining (\ref{7.04.8.19}), (\ref{3.7.7.21}) with the above identity. 
\end{proof}

 We recall from \cite{CK, KR1, Wang2024} the basic set of equations for the null structure.
\begin{align}
&L \bb=-\bb \bar k_{\bN\bN}, \label{lb}\\
&L\tr\chi+\f12 (\tr\chi)^2=-|\chih|^2-{\bar k}_{\bN\bN} \tr\chi-\bR_{44}. \label{s1} \displaybreak[0]
\end{align}

And the following commutation formula holds
\begin{equation}
\f12[L, \Lb]=(\zb^A-\zeta^A) e_A +\f12(\bar k_{\bN\bN} \Lb-\ud k_{\bN\bN} L).\label{3.19.2}
\end{equation}
Using Proposition \ref{ric44} and Lemma \ref{dcom_s}, we derive from (\ref{s1}) that
\begin{lemma}[Normalized structural equations](\cite[Lemma 3.3]{Wang2024})\label{7.03.2.26}
\begin{equation}\label{6.3.1.23}
\begin{split}
 L \tr\chi+\f12 (\tr\chi)^2
&=-|\chih|^2-\widetilde{L \Xi_4}-\tr\chi(\bar k_{\bN\bN}-\f12 \Xi_4)+\N(\Phi, \bp \Phi),
\end{split}
\end{equation}
where
\begin{align}
\widetilde{L \Xi_4}&=L (\Xi_4)+(h-\bar k_{\bN\bN})\Xi_4=\wp\left(\sD \varrho+2\zb^A \sn_A \varrho-\Box_\bg \varrho+(h+\tr\eta)L\varrho\right).\label{3.20.1.22}
\end{align}
\end{lemma}

Next we prove 
\begin{lemma}
Setting $\widehat\Box_\bg f= -(L\Lb f-(\sD f-h\Lb f+h Lf))$ for scalar functions $f$, we have
    \begin{equation}\label{eqn_density}
 -\widehat\Box_\bg \varrho=\frac{\wp}{2}(\Lb \varrho)^2+(\frac{3\wp}{2}-1)|\sn\varrho|^2-(1-\frac{\wp}{2})\sC[\varrho, \varrho]+2c^{-1}\sC[\varrho, v]_\bN-\Box_\bg \varrho,
    \end{equation}
where, for scalar functions $f$, $g$ and vector-field $V$, 
\begin{align*}
\sC[f,  g]=\bg^{\mu\nu}\p_\mu f\p_\nu g; \quad\sC[ f,  V]_\bN=\bg^{\mu\nu}\p_\mu f\p_\nu V_i\c\bN^i.
\end{align*}
Symbolically, 
\begin{equation}\label{eqn_den}
-\widehat\Box_\bg\varrho=\f12\wp(\Lb\varrho)^2+\N(\Phi, \bp\Phi).
\end{equation}

\end{lemma}

\begin{proof}
Using (\ref{8.22.3.25}), we first deduce
\begin{align}\label{8.22.1.25}
\begin{split}
-2 c^{-1}\sC[\varrho,  v]_\bN&=c^{-1}(L\varrho \Lb v_\bN+\Lb\varrho L v_\bN)-2 c^{-1}\sn \varrho \sn v_\bN\\
&=L\varrho\left(\Lb \varrho-(L\varrho+c^{-1}L v_\bN)\right)+c^{-1}\Lb \varrho L v_\bN-2\sn \varrho \zb_A+ 2\sn\log c \sn \varrho
\end{split}
\end{align} 
where we used $\zb_A=c^{-1}\sn_A v_\bN+\sn \log c$. 

Applying (\ref{6.30.2.19}) to $f=\varrho$ and using (\ref{3.22.1.21})-(\ref{1.6.1.21})  and (\ref{8.22.1.25}), we have
\begin{align}
&L\Lb \varrho-(\sD \varrho-h \Lb \varrho+h L\varrho)\nn\\
&=(c^{-1}L v_\bN+L\varrho)L \varrho-(\wp \bN \varrho+c^{-1}L v_\bN) \Lb \varrho+2\zb^A \sn \varrho-\Box_\bg \varrho\nn\\
&=(c^{-1}L v_\bN+L\varrho)L\varrho+\f12\wp (\Lb \varrho)^2-(c^{-1}L v_\bN+\f12\wp L \varrho)\Lb \varrho+2\zb^A \sn \varrho-\Box_\bg \varrho\nn\\
&=\frac{\wp}{2}(\Lb \varrho)^2+(1-\frac{\wp}{2})\Lb\varrho L\varrho+2(\wp-1)|\sn\varrho|^2+2 c^{-1}\sC[\varrho,  v]_\bN-\Box_\bg \varrho,\nn
\end{align}
which implies (\ref{eqn_density}).
(\ref{eqn_den}) then follows immediately in view of  (\ref{4.10.2.19}).
\end{proof}
We are ready to derive important transport equations for the physical radiation field.
\begin{proposition}\label{7.03.3.26}
    Let $z= \tir(\Lb\varrho-h\varrho)$. We have
\begin{align}
Lz&=\tir (\sD \varrho-\widehat\Box_\bg \varrho)+z(c^{-1}\tir^{-1}-h)-\tir\varrho(Lh+h^2),\label{12.15.1.25}\\
L(\bb z)&=\bb\tir(\sD\varrho+\N(\Phi, \bp\Phi))+\bb \left(z(c^{-1}\tir^{-1}-h)-\tir \varrho(Lh+h^2)\right)\nn\\
&+\frac{\wp}{2}h\varrho \bb(z+\tir h\varrho )+\bb z[L\Phi].\label{nml-eqn}
\end{align}
\end{proposition}
\begin{proof}
Using the fact that $cL\tir=1$, we derive
\begin{align*}
L\left((\Lb\varrho-h \varrho)\tir\right)
&=L(\Lb\varrho- h\varrho) \tir+(\Lb\varrho-h\varrho) c^{-1}\\
&=\tir(\sD \varrho-\widehat\Box_\bg \varrho-h \Lb \varrho-Lh\varrho)+(\Lb\varrho-h\varrho) c^{-1}\\
&=\tir (\sD \varrho-\widehat\Box_\bg \varrho)+(\Lb \varrho-h\varrho)(c^{-1}-\tir h)-\tir\varrho(Lh+h^2)
\end{align*} 
which gives (\ref{12.15.1.25}).

Using  (\ref{lb}) and (\ref{3.22.1.21}), we derive
\begin{align*}
L(\bb z)&=\bb Lz+z L\bb=\bb Lz-\bb\bar k_{\bN\bN}z\\
&=\bb Lz+\bb z(\wp \bN \varrho+c^{-1}L v_\bN)\\
&=\bb Lz-\frac{\wp}{2}(\Lb\varrho)^2 \bb\tir +h\varrho \bb \tir(\frac{\wp}{2}\Lb\varrho-\frac{\wp}{2}L\varrho-c^{-1}L v_\bN)+\Lb \varrho \bb \tir(\frac{\wp}{2}L\varrho+c^{-1}L v_\bN). 
\end{align*}
(\ref{nml-eqn}) then follows by straightforward calculation with the help of (\ref{eqn_den}) and (\ref{12.15.1.25}).
\end{proof}
\section{Decay and comparison estimates} \label{decay}
In this section, we first give a preliminary result 
in Proposition \ref{local}, where we present the local energy estimates in $D_0$ and list the values of important geometric quantities of the $u$-foliation at $\Sigma_{t_0}^m$. 
\begin{proposition}\label{local} On the annulus $\Sigma_{t_0}^m$, the following hold
\begin{equation*}
\tir=\bar c^{-1}r, y'=y=0, \bb=\bar c, \theta=0, \la^i=0.
\end{equation*}
Recall that $\Phi=(v^i, \varrho)$ and $t_0=5 \bar c^{-1}$.  Let $l_0\in \mathbb N$ be fixed and sufficiently large. 
 There exists $\ve_0>0$, depending only on $f, g$ in (\ref{wave}), their derivatives and the constant state $\bar \rho>0$,  such that, if $\ve$ in (\ref{wave}) satisfies $0<\ve\le \ve_0$,  the following estimates hold in $D_0$ for $0<t\le t_0$, 
\begin{equation}\label{4.6.0.26}
\|\bp \Phi\|^2_{H^{l+1}(\Sigma_t)}\les \|\bp \Phi\|^2_{H^{l+1}(\Sigma_0)},\,\, l\le l_0, \quad\bp^{\le l_0}\Phi=O(\ve), \quad  c-\bar c=O(\ve).
\end{equation}
\end{proposition}
\begin{proof}
The quantities at $\Sigma_{t_0}^m$ can be derived by the set-up of the $u$-foliation therein. (\ref{4.6.0.26}) can be obtained by the standard local energy estimates for the system (\ref{4.10.1.19}) and (\ref{4.10.2.19}), and Sobolev embedding. We omit the details here for simplicity.
\end{proof}


Note that $\Sigma_{t_0}^m$ is an annulus with $-\bar c^{-1}\le u\le \bar c^{-1}$, i.e. $-1\le q\le 1$, while $q\in [-5, 1]$ on $\Sigma_{t_0}$. Therefore, the annulus lies at the outer $1/3$ of $\Sigma_{t_0}$ upto the exterior boundary $r=r_M$.  This thin annulus scenario matches with the geometric setup for the result of \cite[Theorem 17.1]{Miao_thesis}. Due to Proposition \ref{local}, we can apply the result of \cite[Theorem 17.1]{Miao_thesis} to obtain
\begin{proposition}\label{decay_ricci_coef} There exists a sufficiently small $\ve_0>0$ \begin{footnote}{This $\ve_0>0$ is no greater than the one appearing in Proposition \ref{local}. In the subsequent analysis, we may further shrink $\ve_0>0$ so that the required smallness on $\ve>0$ in  comparison or approximation results can be satisfied. Note that those smallness requirement still only depend on universal upper-bounds. We continue to use the notation $\ve_0$.}\end{footnote} such that, for $\ve$ in (\ref{wave}) satisfying $0<\ve\le \ve_0$, the smooth local solution in $D_0$ admits a unique smooth extension to $D^+$, for $t_0\le t< T_\ve^e$,  satisfying the following properties:
\begin{enumerate}
\item  on $\Sigma_t^m$, the following estimates hold 
\begin{align}
&|c\tr\chi-\frac{2}{\tir}, \chih, \chibh|\les\ve \l t\r^{-2}(\log  t+1)\nn\\
&|\bb \Lb \Phi, \p_u \Phi,  \Phi|\les  \ve \l t\r^{-1}, \quad  \l t\r|\sn\Phi, L\Phi|+|\Omega\Phi|\les\ve  \l t\r^{-1}, \quad |\slashed{\Delta}\Phi|\les \ve \l t\r^{-3} \label{3.21.1.26}\\
&|\bb-\bar c|\les \ve (\log t+1), \quad |\sn\bb|\les \ve \l t\r^{-1}(\log  t+1)\nn\\
&|\bb^2 \bp^2\Phi|\les\ve \l t\r^{-1},\nn
\end{align}
\item
 if $T_\ve^e<\infty$, then $\inf_{\Sigma_t^m} \bb\rightarrow 0$ as $t\rightarrow T_\ve^e$, 
\item the above estimate for $\tr\chi$, together with $\bb>0$, also justifies the transported pullback coordinates used in the geometric setup. 
 Prior to shock formation, every point of $S_{t,u}$ lies on a unique outgoing null geodesic originating from $S_{t_0,u}$, while a shock point is reached as a future endpoint of such a null geodesic.
\end{enumerate}
\end{proposition}
\begin{proof}
The estimates in (1) and (3) can be located in \cite[Theorem 17.1, Section 17.5, Page 450- 451]{Miao_thesis}. Due to (\ref{3.21.1.26}), 
$$
|\bb \bp\Phi, \bb^2 \bp^2\Phi|\les \ve \l t\r^{-1}.$$
 Then, assuming that there is a constant $c_0>0$ such that  $\inf_{\Sigma_t^m}\bb>c_0$ for $t_0\le t<T_\ve^e$, we have $|\bp \Phi, \bp^2\Phi|\les_{c_0} \ve \l t\r^{-1}$. Using the estimates as continuation principle, we can extend the smooth solution $\Phi$ beyond $T_\ve^e<\infty$ by the standard energy argument and the local existence result, leading to a contradiction with the definition of $T_\ve^e$. This gives (2).    
\end{proof}
From now on, we assume $0<\ve\le \ve_0$ where $\ve_0$ is chosen such that Proposition \ref{decay_ricci_coef} holds,   and may be further reduced according to the convention stated there. 
We can derive the following basic control on geometric quantities in $D^+$. 
\begin{proposition}\label{7.9.4.26}
In $D^+$,  we have
\begin{enumerate}
\item the following comparison results,
\begin{equation*}
\tir\approx r\approx \l t\r,\, c-\bar c=O(\ve/\l t\r),
 \end{equation*}
\item the following estimates
\begin{equation}\label{1.2.1.25}
\la=O(\ve\log \l t\r), \quad  r y', \tir y =O(\ve\log \l t\r),
\end{equation}

\item the following comparison results of $\tir$ and $r$, 
\begin{equation}\label{3.25.2.26}
\frac{\bar c\tir}{r}, \frac{r}{\bar c \tir}=1+O(\ve\l t\r^{-1}(\log\l t\r)^2)
\end{equation}
where $\ve>0$ is sufficiently small.
\end{enumerate} 
\end{proposition}
\begin{proof} 
In $D^+$, it is easy to see that  $ \tir\approx \l t\r$ and $4\le r\le r_M\approx \l t\r$.

For any point $p\in D^+$, due to (3) in Proposition \ref{decay_ricci_coef}, there is a unique outgoing null geodesic, originating from $(t_0, \omega, q)$ such that
\begin{equation}\label{7.31.2.26}
 p=\Upsilon_{\omega, q}(t). 
 \end{equation}
Using $cL(c^{-1})=-L\log c$ and (\ref{3.21.1.26}), integrating along the null geodesic $\Upsilon_{\omega,q}$ in (\ref{7.31.2.26}) gives
$$c^{-1}(t)-c^{-1}(t_0)=O(\ve).$$ In view of the last estimate in (\ref{4.6.0.26}), we have $c^{-1}=\bar c^{-1}+O(\ve)$ in $D^+$. Hence $c=\bar c+O(\ve)$ in $D^+$.  This estimate will be improved shortly. 

Next, we claim
\begin{equation}\label{7.24.1.26}
\b \tir^{-1}=O(\ve) \mbox{ in } D^+. 
\end{equation}
Indeed,  rewriting the transport equation (\ref{1.22.4.22}) for $\b$ implies
\begin{equation*}
\sn_L \b-\chi\c \b=\bb(\zb-\ze).
\end{equation*}
Using
 Proposition \ref{decay_ricci_coef}, (\ref{5.7.1.26}), $\b(t_0)=0$ and $c-\bar c=O(\ve)$, we can obtain (\ref{7.24.1.26}) by integrating along $\Upsilon_{\omega, q}(t)$ from $\Sigma_{t_0}^m$. 
 
Note $
\p_u f=\bb\bN f+\b(f).$
Then using (\ref{3.21.1.26}), (\ref{7.24.1.26})   and $c-\bar c=O(\ve)$,
 we have 
$$|c-\bar c|=|\int_u^{\bar c^{-1}} \p_u c|=O(\ve \l t\r^{-1}).$$
This is the last estimate in (1). We postpone proving $r\approx \tir$ in (1).

 Next we consider (2).   Using $c-\bar c=O(\ve)$, (\ref{3.22.5.21}) and (\ref{3.21.1.26}), for $p\in D^+$ we have by integrating along $\Upsilon_{\omega,q}(t)$ from $t=t_0$ that 
\begin{align*}
|\la|\les \int_{t_0}^t (|\Omega\Phi|+|v|)\les\ve\log \l t\r.
\end{align*}
Using  $r y'\approx\la$ from (\ref{12.20.3.21}),
 we can obtain the estimate of $|y'|$.

Using $\Phi=O(\ve\l t\r^{-1})$  and $c-\bar c=O(\ve \l t\r^{-1})$ in $D^+$, in view of (\ref{8.17.1.22}), for $p\in D^+$ we deduce by integrating along $\Upsilon_{\omega, q}(t)$ that
\begin{align*}
    |r(t,\omega, r_0)-\bar c (t+u)|&\les\int_{t_0}^t \{|c-\bar c|+| v(\bN)|+|y'|\} dt'\\
    &\les \ve\log \l t
    \r+\int_{t_0}^t |\la| /r\\
    &\les \ve\log \l t\r+\int_{t_0}^t \frac{|\la|}{\bar c \tir}(|\frac{\bar c \tir}{r}-1|+1) dt'.
 \end{align*}
 With $\fS=\frac{r}{\bar c \tir}-1$ and noting that $\frac{\bar c \tir}{r}-1=-\frac{\fS}{\fS+1}$, using the $\lambda$-estimate in (\ref{1.2.1.25}), we derive
 \begin{align*} 
|\fS|&\les \l t\r^{-1} (\log \l t \r)^2 \ve\left(1+\sup_{t_0<t'<t}|\bar c\tir/r-1|\right)\\
&\les \l t\r^{-1}(\log \l t\r)^2 \ve\left(1+\sup_{t_0<t'<t}|\frac{\fS}{\fS+1}|\right).
\end{align*}
Since $\fS(t_0)=0$, with an  auxiliary  bootstrap assumption that $|\fS|\le \frac{2}{3} $ in $D^+$, we can obtain $\fS=O(\l t\r^{-1}(\log \l t\r)^2 \ve)$, which improves the auxiliary  bootstrap assumption with $\ve>0$ sufficiently small. This also implies $r\approx \tir$ and
\begin{equation*}
\frac{\bar c \tir}{r}-1=O(\ve \l t\r^{-1}(\log \l t\r)^2).
\end{equation*}
 Thus (\ref{3.25.2.26}) is proved. 
 
 To derive the estimate $\tir y$, we integrate (\ref{4.7.1.26}) along $\Upsilon_{\omega, q}(t)$, with the data $y(t_0)=0$ in Proposition \ref{local}. The result follows then by using (\ref{3.21.1.26}). 
 
\end{proof}
Next we prove the following result about the variations of $q$ and $\omega$ along an outgoing null geodesic  which starts at $\Sigma_{t_0}^m$.
\begin{corollary}\label{4.17.1.26} 
Along an outgoing null geodesic $\Upsilon_{\omega_0, q_0}(t)$ starting from $\Sigma_{t_0}^m$ with $\omega_0\in {\mathbb S}^2$ and $q_0\in [-1,1]$, we have in $D^+$ that
\begin{equation}\label{4.17.2.26}
    q-q_0=O(\ve (\log(1/\ve))^2), \quad \omega-\omega_0=O(\ve).
\end{equation}
\end{corollary}
\begin{proof}
Note that at $t=t_0$, $q_0=\bar c u$, 
$$q-q_0=r-\bar c t-\bar c u=r-\bar c \tir.$$ Then the first estimate follows immediately from using (\ref{3.25.2.26}). 

To derive the second one, we first write 
\begin{equation*}
\Upsilon_{\omega_0, q_0}(t)=\left(t, x(t)\right)=\left(t, r(t)\omega(t)\right).
\end{equation*}
Differentiating along the null geodesic yields
\begin{equation*}
  cL=\frac{d}{dt}(t,  x(t))=(1, \dot x(t))=\p_t+v^i\p_i+c\bN^i\p_i.
\end{equation*}
This implies
\begin{equation*}
\dot x(t)=v^i+c \bN^i. 
\end{equation*}
Due to $x=r\omega$
\begin{equation*}
   \frac{d}{dt}(r(t) \omega(t))=\dot r(t)\omega(t)+r(t)\dot\omega(t). 
\end{equation*}
Hence in view of (\ref{8.17.1.22}) we have
\begin{align*}
r\dot\omega&=c\bN^i+v^i-\dot r\omega=c\bN^i+v^i-c(c^{-1}v(\bN)+{y'}^k(\frac{ x^k}{r}-c^{-1}v^k)+1)\omega\\
&=c(\bN-\omega)+v-(v(\bN)+c y'\c(\omega-c^{-1} v))\omega,
\end{align*}
which gives
\begin{equation*}
\dot\omega=c \frac{y'}{r}+\frac{1}{r}(v^i-v(\bN) \frac{x^i}{r})-c\frac{y'}{r}(\omega-c^{-1}v)\omega.
\end{equation*}
Hence, also noting $v= O(\l t\r^{-1} \ve)$ from (\ref{3.21.1.26}) and 
\begin{equation*}
|\dot \omega|\les |y'/r|+|v/r|
\end{equation*}
we derive by using (\ref{1.2.1.25}) that
\begin{equation*}
|\omega(t)-\omega(t_0)|\les \ve .
\end{equation*}
Therefore the proof is complete. 

\end{proof}

  \section{Proof of Theorem \ref{main}
 }\label{pmain}
 Throughout this section, we suppose $0<\ve\le \ve_0$, where $\ve_0$ is chosen such that Proposition \ref{decay_ricci_coef} holds and may be further reduced as needed, as explained there. All such reductions depend only on universal upper bounds.
 
 We divide the proof of Theorem \ref{main} into two steps: 
 \begin{enumerate}
   \item Using the results in Section \ref{prl_est}, the inverse function theorem and Bernoulli's law, we show Proposition \ref{12.31.1.25}, which says that at $t_1=\frac{1}{2\ve}$ a certain compression occurs. This gives the desired initial data at the intermediate time $t_1$ to form blow-up within finite time by virtue of the Riccati-equation mechanism given in Proposition \ref{Riccati_prop};
 \item  Prove Proposition \ref{Riccati_prop} by using Proposition \ref{7.03.3.26} and Proposition \ref{decay_ricci_coef}; using this result, we complete the proof of Theorem \ref{main} by using Proposition \ref{12.31.1.25}.
 \end{enumerate}
 
  \subsection{Formation of compression} We first give the result of formation of compression at $t_1=\frac{1}{2\ve}$ along $\Upsilon_*$. Recall that we denote by $\Upsilon_{\omega_0, q_0}$ the outgoing null geodesic starting from $(t_0, \omega_0, q_0)$, with $\omega_0\in \mathbb S^2$ and $q_0\in [-1, 1]$.
\begin{proposition}\label{12.31.1.25}
 Let $\ti z= c z$ be the physical radiation field, first introduced in (\ref{eq:phys.radfield}).
Along an outgoing null geodesic $\Upsilon_{\omega_0, q_0}(t)$ starting from $(t_0, \omega_0, q_0)$, with $\omega_0\in \mathbb S^2$ and $q_0\in[-1,1]$, we have, with $t_1=\frac{1}{2\ve}$, 
\begin{equation}\label{key_apx}
    \ti z\big(\Upsilon_{\omega_0, q_0}(t_1)\big)=-2\ve\bar c^{-1}\p^2_q F_0(\omega_0, q_0)+O\left((\ve\log(1/\ve))^2\right). 
\end{equation}
\end{proposition}
We divide the proof of Proposition \ref{12.31.1.25}  into three steps:
\begin{enumerate}
\item we first prove a comparison result between $\ti z$ with $\p_q(r\varrho)$ in Lemma \ref{comp_radiation}; 
\item then we derive the crucial approximation result in Proposition \ref{7.03.1.26}, which approximates $\p_q(r\varrho)(t,x)$ by the leading term $\ve \bar c^{-1} \p_q^2 F_0(\omega, q)$, where $F_0$ is the Friedlander  radiation field introduced in Theorem \ref{known};
\item finally, using the result of Corollary \ref{4.17.1.26} which justifies, along the outgoing null geodesic $\Upsilon_{\omega_0, q_0}$, $\p_q^2 F_0(\omega, q)$ can be approximated by $\p_q^2 F_0(\omega_0, q_0)$, we  complete the proof of Proposition \ref{12.31.1.25} by establishing Corollary \ref{12.31.7.25}.
\end{enumerate}  

Let us first control the difference between $-c\tir(\Lb\varrho-h\varrho)$ with $\p_q(r\varrho)$ at $t_1=\frac{1}{2\ve}$. For this purpose, we note $\stc \Lb=\bar c^{-1}\p_t-\p_r=-2\p_q$.
\begin{lemma}\label{comp_radiation}
     At $\Sigma_{t_1}^m$, we have
    \begin{equation*}
\ti z-\stc\Lb(r \varrho)=O(\ve^3(\log (1/\ve))^2). 
    \end{equation*}
\end{lemma}
\begin{proof}
    It is easy to see $\stc \Lb r=-1$. Moreover
    \begin{equation*}
    \Lb=\Lb^i \p_i +\Lb^0 \p_t=c^{-1} \p_t+(c^{-1}v^i-\bN^i)\p_i.
    \end{equation*}
Hence
\begin{align*}
\ti z-\stc\Lb(r \varrho)&=c \tir (c^{-1}\p_t +(c^{-1}v^i-\bN^i)\p_i)\varrho-r(\bar c^{-1}\p_t-\p_r)\varrho-\varrho(c\tir h-1)\\
&=(\tir -\bar c^{-1}r)\p_t \varrho+(\tir v^i \p_i+r\p_r-c\tir \bN^i\p_i )\varrho-\varrho(c\tir h-1)\\
&=(\tir -\bar c^{-1}r)\p_t \varrho+\left(\tir v^i +c\tir(\frac{x^i}{c\tir}-\bN^i)\right)\p_i \varrho-\varrho(c\tir h-1)\\
&=(\tir-\bar c^{-1}r)\p_t\varrho+(\tir v^i-c\tir y^i)\p_i \varrho-\varrho(c\tir h-1).
\end{align*}
Using Proposition \ref{12.31.4.25}, (\ref{3.21.1.26}), (\ref{1.2.1.25}) and (\ref{3.25.2.26}), we have 
\begin{equation*}
\ti z-\stc\Lb(r \varrho)=O(\ve^3(\log (1/\ve))^2)
\end{equation*}
as desired. 

\end{proof}
 


Then we prove the following important result of approximation. 
\begin{proposition}\label{7.03.1.26}
Let $\f12 t_1< t< 2 t_1$. With $\ve>0$ sufficiently small, we have at $\Sigma_t^m$ 
    \begin{equation}\label{3.23.1.26}
r\varrho
=
\varepsilon \bar c^{-1} \p_q F_0
-
\varepsilon^2 r^{-1} A_2 (\p_q F_0)^2
+\mathcal E(t,x),
    \end{equation}
with
$A_2=\frac{1}{\bar c^2}+\frac{w''(\bar\rho)\bar\rho^2}{2\bar c^4}$ and
\begin{equation}
|\mathcal E(t,x)| + |\partial_q \mathcal E(t,x)|\les\varepsilon^2\log(1/\ve).
\label{3.23.2.26}
\end{equation}
\end{proposition}

\begin{proof}
Recall from Proposition \ref{apx_prop} that for
$$
\phi_\varepsilon(t,x) = \varepsilon r^{-1}F_0(\omega, q) + \mathcal R_\varepsilon(t,x),
$$
we have
\begin{equation}
|\mathcal R_\varepsilon| + |\p^\a \mathcal R_\varepsilon|
\les \varepsilon r^{-2}, |\a|=1,2.
\label{3.23.3.26}
\end{equation}
In view of the Bernoulli formula (\ref{Bernoulli}), we have 
\begin{equation}\label{12.31.3.25}
w(\rho)=-(\p_t \phi_\ve+\f12|\nab \phi_\ve|^2)+\E_{\mbox{\tiny Bern}}
\end{equation}
with the following estimates obtained by using Proposition \ref{12.31.4.25}, 
\begin{equation}\label{12.31.5.25}
|\E_{\mbox{\tiny Bern}}, \bp \mathcal E_{\mathrm{Bern}}|\les \ve^2 \log(1/\ve) r^{-1}.
\end{equation}
Here we used $r\approx t$ in $D^+$ from Proposition \ref{7.9.4.26}. 

Using $q=r-\bar c t$,
we deduce
$$
\partial_t(\varepsilon r^{-1}F_0)
=
-\varepsilon \bar c r^{-1}\p_q F_0.$$

Due to (\ref{3.23.3.26}), we have
$$
-\partial_t\phi_\varepsilon
=
\varepsilon \bar c r^{-1}\p_q F_0
+
O^{\le 1}(\varepsilon r^{-2}).
$$

Similarly for the gradient term, we derive
$$
|\nabla\phi_\varepsilon|^2
=
\varepsilon^2 r^{-2}(\p_q F_0)^2
+
O^{\le 1}(\varepsilon^2 r^{-3}).$$
Hence, in view of (\ref{12.31.3.25}), we obtain
\begin{equation}\label{3.23.4.26}
w(\rho)
=
\varepsilon \bar c r^{-1}\p_q F_0
-
\frac12 \varepsilon^2 r^{-2}(\p_q F_0)^2
+
O^{\le 1} \!\big(\varepsilon^2 r^{-1}\log (1/\ve)\big),
\end{equation}
where $\rho=\bar\rho e^{\varrho}$ is the original density function introduced in Section \ref{setup}.
Assuming the following estimate,
\begin{equation}\label{3.23.5.26}
\varrho
=\frac{w}{\bar c^2}
-(\frac{1}{2\bar c^4} 
+\frac{w''(\bar\rho)\bar\rho^2}{2\bar c^6}) w^2
+E
\end{equation}
where the error term $E=O^{\le 1}_q(\ve^3 r^{-3})$, 
we conclude (\ref{3.23.1.26}). 

Next we prove (\ref{3.23.5.26}).
 Due to $w(\bar\rho)=0$, the enthalpy can be regarded as the smooth composite function
$$ \mathfrak{W}(\varrho):=w(\bar\rho e^{\varrho})-w(\bar\rho)=w(\rho).
$$
 By an abuse of notation, we henceforth write $w(\varrho)$ for $\mathfrak{W}(\varrho)$. 
Since $\bar\rho>0$ and $w$ is smooth in a neighborhood of $\bar\rho$, this function is smooth for $\varrho$ sufficiently small and can be expanded at $\varrho=0$.
To derive such an expansion, we expand the enthalpy $w$ near $\bar \rho$, with $\delta\rho=\rho-\bar \rho$, 
\begin{equation*}
    w(\rho)=w'(\bar\rho)\delta\rho+\f12 w''(\bar\rho)(\delta \rho)^2+(\delta \rho)^3 R_1(\rho),
\end{equation*}
where 
\begin{equation*}
|R_1|\les 1, |\p_q((\delta \rho)^3 R_1)|\les  \ve^3 r^{-3}
\end{equation*}
and $w'(\bar\rho)=\bar c^2/\bar\rho$, where we used $|\varrho|+|\p_q\varrho|\les \ve r^{-1}$ due to Proposition \ref{12.31.4.25}.
 
In view of 
$
\bar\rho^{-1}(\rho-\bar\rho)= \varrho +
\frac12  \varrho^2+
\varrho^3 R_0(\varrho)
$
with $|R_0(\varrho)|\les 1$ and $|\p_q(\varrho^3R_0)|\les \ve^3 r^{-3}$, 
we have
\begin{equation}
w=\bar c^2 \varrho
+
\frac12\big(
\bar c^2 + w''(\bar\rho)\bar\rho^2
\big)\varrho^2
+
R_2,\label{3.23.6.26}
\end{equation}
where $R_2=O^{\le 1}_q(\ve^3 r^{-3})$. 
Since for the function $w(\varrho)$, we have
$$w(0)=0,
\quad
w'(0)=\bar c^2\neq 0,$$
the inverse function theorem implies $w(\varrho)$ is locally invertible in a neighbourhood of $\varrho=0$. Hence there exists a smooth local inverse  $\varrho(w)=\Theta(w)$, defined in a neighbourhood of $w=0$. 
Then we expand $\Theta(w)$ at $w=0$ in powers of $w$, yielding 
  $$\varrho = a\,w + b\,w^2 + E,$$
with $|E(w)|\les |w|^3$ and $a, b$ constants.  Higher order expansion can be  achieved similarly.

To determine the constants $a$ and $b$, we substitute the above formula into (\ref{3.23.6.26}) and matching the coefficient, which yields
$$a=\bar c^{-2}, b = -\frac{B}{\bar c^6},$$
where $ B=\frac12\big(
\bar c^2 + w''(\bar\rho)\bar\rho^2
\big).$
   To obtain the derivative control of $E$, we
substitute the obtained expansion to (\ref{3.23.6.26}) again, which implies
\begin{equation*}
w=
\bar c^2\Big(\frac{w}{\bar c^2} - \frac{B}{\bar c^6} w^2 + E\Big)
+
B\Big(\frac{w}{\bar c^2} - \frac{B}{\bar c^6} w^2 + E\Big)^2
+R_2.
\end{equation*}
Hence with $C_1$ and $C_2$ two constants depending on the constant states, we have
\begin{equation*}
\bar c^2 E+B(E^2+2E(\frac{w}{\bar c^2} - \frac{B}{\bar c^6} w^2 ))+C_1w^3+C_2 w^4+R_2=0.
\end{equation*}
Differentiating the above identity, schematically, we deduce that
\begin{equation*}
(\bar c^2+2B E+O(w)) \p_q E=\p_q( w^3+w^4)+E\p_q(w+w^2)+\p_q R_2.
\end{equation*}
Using Proposition \ref{12.31.4.25}, (\ref{Bernoulli}) and $r\approx \l t\r$ in $\Sigma_t^m$ with $t<1/\ve$, we have
$$
w = O(\varepsilon r^{-1}),\quad \p_q w = O(\varepsilon r^{-1}).$$
 Hence with $\ve$ sufficiently small, we derive by using the error control in (\ref{3.23.6.26})
\begin{equation*}
\p_q E=O(\ve^3 r^{-3}).
\end{equation*}
Thus we proved (\ref{3.23.5.26}).  The proof is therefore complete. 
\end{proof}

\begin{corollary}\label{12.31.7.25}  
    With $\Upsilon_{\omega_0, q_0}(t)$ the outgoing null geodesic starting from $\Sigma_{t_0}^m$, at $t_1=\frac{1}{2\ve}$, we have
    \begin{equation*}
\bar c\p_q (r\varrho)(\Upsilon_{\omega_0, q_0}(t_1))=\ve \p_q^2 F_0(\omega_0, q_0)+O((\ve\log(1/\ve))^2). 
\end{equation*}
\end{corollary}
\begin{proof}
Let $(t_1,x)=\Upsilon_{\omega_0, q_0}(t_1)$, with $x=r\omega$.  It is straightforward to derive from (\ref{3.23.1.26}) that,
$$
\partial_q(r\varrho)(t_1 ,x)
=
\frac{\varepsilon}{\bar c}\,\p_q^2 F_0(\omega, q)
-
\frac{\varepsilon^2}{r}\,A_2\,\partial_q\!\big((\p_q F_0)^2\big)(\omega, q)
+
R(t_1,x),$$
where
$A_2=\frac{1}{\bar c^2}+\frac{w^{''}(\bar\rho)\bar\rho^2}{2\bar c^4},$
and the remainder satisfies
$$
|R(t_1,x)| \les\varepsilon^2\log(1/\ve).
$$
Hence  we have
 \begin{equation*}
\bar c\p_q (r\varrho)(t_1, x)=\ve \p_q^2 F_0(\omega, q)+O(\ve^2\log(1/\ve)). 
\end{equation*}
Corollary \ref{12.31.7.25} then follows by using Corrollary \ref{4.17.1.26} and that $\p_q^2 F_0$ is smooth and compactly supported.
\end{proof}
Proposition \ref{12.31.1.25} follows directly by using Lemma \ref{comp_radiation} and the above result. Hence we complete the proof of Proposition \ref{12.31.1.25}.

\subsection{Proof of Theorem \ref{main}}
 Recall that the smooth solution extends until $T_\ve$. To study the exterior region $D^+$, we recall 
 $$T^e_\ve:=\sup\{t: \mbox{the exterior solution with } u\in [-\bar c^{-1}, \bar c^{-1}] \mbox{ is smooth in } [t_0, t)\}.$$
Clearly $T^e_\ve\ge T_\ve$.   We regard the timespan of $D^+$ as $[t_0, T_\ve^e)$ without mentioning. 

 In Proposition \ref{Riccati_prop}, we first derive a Riccati-type equation in (\ref{Riccati_ap}). 
 Moreover we give in (\ref{5.6.1.26})  good control of $L(\bb z)$ in $D^+$ even when $t\rightarrow {T^e_\ve}_-$. Both of them are crucial for the proof of Theorem \ref{main}.

\begin{proposition}\label{Riccati_prop}

 Along the outgoing null geodesic $\Upsilon_{\omega_0, q_0}$ starting from $\Sigma_{t_0}^m$, we have
\begin{equation}\label{Riccati_ap}
\frac{d}{dt}\ti z(t)=\f12\wp\tir^{-1}\ti z^2+\Er_1(t) \ti z(t)+\Er_0(t), \mbox{ in }D^+
\end{equation}
 with 
\begin{equation*}
 \Er_1(t)=O\left(\ve \l t\r^{-2}(\log t+1)\right),\, \Er_0(t)=O(\l t\r^{-2}\ve),
\end{equation*}
and the constants in the error control are universal. 

Moreover, we have in $D^+$
\begin{equation}\label{5.6.1.26}
L(\bb z)=O(\ve \log \l t\r\l t\r^{-2}).
\end{equation}
\end{proposition}

\begin{proof}
Using (\ref{eqn_den}) and (\ref{12.15.1.25}), we derive, with $\dot{}$ denoting $\p_t$ along the null geodesic, that
    \begin{align}
    cL(cz)&=\dot{c}z+c\dot{z}=\dot{c}z+c^2 \left(\tir (\sD \varrho-\widehat\Box_\bg\varrho)+z(c^{-1}\tir^{-1}-h)-\tir\varrho(Lh+h^2)\right)\nn\\
    &=\tir c^2\left(\f12 \wp (\Lb\varrho)^2+\N(\Phi, \bp\Phi)+\sD\varrho-z\tir^{-1}(h-\frac{1}{c\tir})- \varrho(Lh+h^2)\right)+\dot{c} z\nn\\
    \displaybreak[0]
    &=\tir c^2\big(\f12 \wp(\Lb\varrho-h\varrho)^2+\wp h \varrho(\Lb\varrho-h\varrho)+\f12\wp( h\varrho)^2+\N(\Phi, \bp\Phi)+\sD\varrho\nn\\&-\varrho (Lh+h^2)\big)
    +z\left(\dot{c}-c^2(h-\frac{1}{c\tir})\right)\nn\\
    &=\f12\tir^{-1}\wp\ti z^2+\ti z(c^{-1}\dot{c}-c(h-\frac{1}{c\tir})+c\wp   h \varrho)+\tir c^2\big(\f12\wp (h\varrho)^2+\N(\Phi, \bp\Phi)\nn\\
    &+\sD\varrho-\varrho (Lh+h^2)\big).\nn
    \end{align}
         Using (\ref{4.10.2.19}), (\ref{7.04.8.19}) and (\ref{6.3.1.23}), we write
    \begin{equation*}
    2(Lh+h^2)=-|\chih|^2-\tr\chi\c [L\Phi]+\N(\Phi, \bp\Phi)+\sD\varrho.
    \end{equation*}
           Note
    \begin{align*}
    \N(\Phi, \bp\Phi)&=(\Lb\varrho-h\varrho)\c \bar\bp\Phi+(h \varrho \c \bar\bp\Phi +\bar\bp \Phi\c \bar\bp\Phi).
\end{align*}
   
        It remains to estimate the error terms, which schematically take the following form 
        \begin{align*}
        \Er_1&=(h-\frac{1}{c\tir})+\wp   h \varrho+(\varrho+1) \bar \bp\Phi,\\ \Er_0&=\tir \big( \wp (h\varrho)^2
        +(\varrho+1)\sD\varrho+\varrho\big(|\chih|^2+\tr\chi[L\Phi])\big)+\tir(1+\varrho)(h \varrho  \bar\bp \Phi+\bar \bp \Phi\c \bar \bp \Phi),
        \end{align*}
     where the factors given by powers of $c$ are dropped in the above error terms.
     
        Using (\ref{3.21.1.26}), we deduce that, in $D^+$
\begin{align*}
    |\Er_1|&\les \ve \l t\r^{-2}(\log t +1), \quad|\Er_0|\les \ve \l t\r ^{-2}.
\end{align*}

Similar to the above estimates, we derive by using (\ref{3.21.1.26}) that 
\begin{align*}
&\bb \N(\Phi, \bp\Phi)=O(\ve^2 \l t\r^{-3}),\quad \bb(Lh+h^2)=O(\ve \l t\r^{-3}\log \l t\r) \\
&\bb z(c^{-1}\tir^{-1}-h)=O(\ve^2\log \l t\r \l t\r^{-2})\\
&h\varrho\bb(z+\tir h\varrho)=O(\ve^2 \l t\r^{-2}), \quad\bb z[L\Phi]=O(\ve^2 \l t\r^{-2}). 
\end{align*}
Hence in view of (\ref{nml-eqn}) and (\ref{3.21.1.26}) we conclude $L(\bb z)=O(\ve\log \l t\r \l t\r^{-2})$ as stated in (\ref{5.6.1.26}).

\end{proof}

Applying Proposition \ref{12.31.1.25} to the outgoing null geodesic $\Upsilon_*(t)$, initiated at $(\omega_*, q_*, t_0)$, at $t=t_1$, the first term on the right-hand side of (\ref{key_apx}) is $-2\ve \bar c^{-1}\p_q^2 F_0(\omega_*, q_*)$ which is positive and $\approx \ve$. 
With $\ve$ sufficiently small, this leads to
\begin{equation*}
\ti z(t_1)-\K>0
\end{equation*}
where 
$$
\mathcal K=\int_{t_1}^T |\Er_0(t)| dt \exp\int_{t_1}^T |\Er_1(t)| dt=o(\ve)
$$
with $\Er_0$ and $\Er_1$  the error terms in Proposition \ref{Riccati_prop}, and the integral is along $\Upsilon_*(t)$. 

Similar as \cite{Alinhac, Yin}, using the standard ODE argument \cite[Page 7, Lemma 1.3.2]{Hormander}, if $\ti z(t)$ is a solution of (\ref{Riccati_ap}) on $[0, T]$, 
  then \begin{footnote}{$2-$ denotes a fixed constant sufficiently close to $2$ from below.}\end{footnote}
 \begin{align}\label{5.12.1.26}
 \f12 \wp\left(\log(T+u)-\log (u+t_1)\right)&<(\ti z(t_1)-\K)^{-1} \exp(\int_{t_1}^T |\Er_1(t)|dt)\nn\\
 &\les\big(-2\ve \bar c^{-1}\p_q^2 F_0(\omega_*, q_*)+o(\ve)\big)^{-1} \exp(O(\ve^{2-})). 
 \end{align}
 Hence we conclude from (\ref{5.12.1.26})
the lifespan of $D^+$ satisfies 
$$\limsup_{\ve\rightarrow 0}\ve \log T_\ve^e\le \tau_*, $$
and thus (\ref{5.10.1.26}) holds true due to $T_\ve\le T_\ve^e$ and Theorem \ref{known}. 
 Using $T^e_\ve\ge T_\ve$ again and Theorem \ref{known}, we have 
  $$\liminf_{\ve\rightarrow 0} \ve \log T^e_\ve\ge \tau_*.
  $$
Hence we conclude (\ref{5.13.3.26}). 
Moreover using Proposition \ref{decay_ricci_coef} (2), due to $0<T_\ve^e<\infty$,
\begin{equation*}
\inf_{\Sigma_t^m}\bb \rightarrow 0, \mbox{ as } t\rightarrow T_{\ve-}^e. 
\end{equation*}

Finally we prove the blow-up of $\bp \Phi$ as $t\rightarrow T_{\ve-}^e$. 

Let $T_*$ be the first time when the solution $\ti z$ of (\ref{Riccati_ap}) blows up along $\Upsilon_*$. $T_*$ has the upper bound given by (\ref{5.12.1.26}).
 Clearly $T_\ve^e\le T_*$. 
 Next we prove at the point of shock formation which is at  $t= {T_\ve^e}_-$, $z$ blows up.   

Using (\ref{lb}), (\ref{3.22.1.21}) and (\ref{3.21.1.26}), integrating along an outgoing null geodesic $\Upsilon(t)$ starting from $\Sigma_{t_0}^m$,  we derive
$
\bb(t_1)-\bb(t_0)=O(\ve)log \l t_1\r,
$
which implies, in view of $\bb(t_0)=\bar c$ from Proposition \ref{local} and with $\ve$ sufficiently small,  that
\begin{equation}\label{7.21.3.26}
\bb(t_1)=\bar c+O(\ve \log (1/\ve))\approx \bar c.
\end{equation}
Using (\ref{5.6.1.26}), integrating along $\Upsilon(t)$ from $t_1$ to $t\in (t_1, T_\ve^e)$, we have    
\begin{equation}\label{7.21.2.26}
\bb z(t)-\bb z(t_1)=O(\ve^{2-}). 
\end{equation}

Using (\ref{lb}) and (\ref{3.22.1.21}), we derive
\begin{equation}\label{7.22.2.26}
L\bb=-\f12\wp \bb z\tir^{-1}+\bb[L\Phi]-\f12\bb h \varrho.
\end{equation}
In view of (3) in Proposition \ref{decay_ricci_coef}, we suppose shock forms along an outgoing null geodesic $\Upsilon_1(t)$, starting from $\Sigma_{t_0}^m$,  along the null cone $\H_{u_1}$ as $t\rightarrow T_\ve^e$, and suppose 
$z$ is bounded  as $t\rightarrow T_\ve^e$ along $\Upsilon_1$, then $\bb z({T_\ve}_-^e)=0$. Due to (\ref{7.21.2.26}), 
we have 
\begin{equation}\label{7.21.4.26}
\bb z(\Upsilon_1(t_1))=O(\ve^{2-}).
\end{equation}
Integrating (\ref{7.22.2.26}) along $\Upsilon_1$ yields
\begin{equation*}
\bb(\Upsilon_1(t_1))=-\int_{t_1}^{T_\ve^e}c L\bb dt'=\int_{t_1}^{T_\ve^e}\{ \f12 \wp c \bb z \tir^{-1}-c\bb[L\Phi]+\f12 c \bb h\varrho\}.
\end{equation*} 

Using $c-\bar c=O(\ve \l t\r^{-1})$ due to Proposition \ref{7.9.4.26} (1), and using (\ref{3.21.1.26}), (\ref{7.21.2.26}) and (\ref{7.21.3.26})
\begin{align}\label{7.26.2.26}
\f12\wp \int_{t_1}^{T_\ve^e}\tir^{-1}\bar c (\bb z(\Upsilon_1(t_1))+O(\ve^{2-}))\ges \bar c.
\end{align}
Hence
\begin{align}\label{8.5.1.26}
\f12 \wp|\bb z(\Upsilon_1(t_1))+O(\ve^{2-})|\left(\log (T_\ve^e+u_1)-\log(t_1+u_1)\right)\ges 1.
\end{align}
Using the fact that $T_\ve^e\le T_*$, we have (\ref{5.12.1.26}) holds for $T_\ve^e$.
Noting that with $F(u)=\log (\frac{T_\ve^e+u}{t_1+u})$, it is direct to estimate
\begin{equation*}
F(u_1)-F(u)=O(\ve|u-u_1|).
\end{equation*}
Hence, in view of (\ref{5.12.1.26}) and (\ref{7.21.4.26}), we deduce from (\ref{8.5.1.26}) that
\begin{equation*}
\left((-2 \ve \bar c ^{-1} \p_q^2 F_0(\omega_*, q_*)+o(\ve))^{-1}+O(\ve) \right)|\bb z(\Upsilon_1(t_1))+O(\ve^{2-})|\ges 1,
\end{equation*}
which implies
\begin{equation*}
|\bb z(\Upsilon_1(t_1))|\ges \ve.
\end{equation*} 
If $ \bb z(\Upsilon_1(t_1))\les -\ve$, the integral in (\ref{7.26.2.26}) is negative, which is impossible.
Therefore
\begin{equation*}
\bb z(\Upsilon_1(t_1))\ges \ve,  \begin{footnote}{
  Due to (\ref{7.21.3.26}), this
actually implies $z(\Upsilon_1(t_1))\ges \ve$  i.e. compression formed already at $\Upsilon_1(t_1)$.}\end{footnote} 
\end{equation*}
  contradicting  (\ref{7.21.4.26}).
 We conclude 
$|z|$ blows up  when approaching the first shock  along $\Upsilon_1(t)$. 
 
 Then, using (\ref{3.21.1.26}) again, we have 
$$\limsup_{t\rightarrow T_{\ve-}^e}\tir|\Lb \varrho|=\infty.$$ Using (\ref{8.22.3.25}) and (\ref{3.21.1.26}),  we have $\limsup_{t\rightarrow T_{\ve-}^e}\tir|\Lb v_\bN|= \infty$  when approaching the shock along $\Upsilon_1$. Consequently, due to $\tir\approx \l t\r$, in view of (\ref{3.21.1.26}) and $c-\bar c=O(\ve)$,  $$\limsup_{t\rightarrow T_{\ve-}^e}|\bB\Phi|=\infty$$ approaching the shock.
  Therefore the proof of Theorem \ref{main}  is complete.
  
\section{Appendix}\label{Apdx}
\subsection{Calculation of $G(\omega)$}
We first check the result 
\begin{proposition}\label{G-value}
    The quantity $G(\omega)=-\frac{1+\ga}{\bar c}$ in Theorem \ref{known}. 
\end{proposition}
\begin{proof}
We rely on the method of \cite[Chapter 6]{Hormander} to calculate $G(\omega)$. 
In the sequel our notation is consistent with the convention in \cite[Page 123]{Hormander}.
$$
G(\omega)=\sum_{j,k,l=0}^3 g^{jkl}\,\hat\omega_j\hat\omega_k\hat\omega_l,$$
with the characteristic covector
$\hat\omega=(-\bar c,\omega)$ \begin{footnote}{Since 
$q=r-\bar c t,$
the characteristic covector is
$\hat\omega=(-\bar c,\omega)$.}\end{footnote}, the main nonlinear term in the asymptotics coming from the linear-in-$\bp\phi$ part of the quasilinear coefficients.

For the irrotational Euler potential equation, expanding around the constant state and keeping only the terms linear in first derivatives gives
\begin{equation}\label{12.30.1.25}
\frac{1}{\bar c^2}\phi_{tt}-\Delta\phi
+\frac{2}{\bar c^2}\sum_{i=1}^3 \phi_i\,\phi_{ti}
+\frac{\gamma-1}{\bar c^2}\phi_t\,\Delta\phi
=\text{cubic terms}.
\end{equation}
Indeed, 
in the small amplitude setting,  there holds
\begin{equation*}
\begin{split}
    c^2-\bar c^2
    &=-2\bar \rho\frac{c'(\bar \rho)}{\bar c}\p_t \phi+O(|\bp \phi|^2)=-(\ga-1) \p_t \phi+O(|\bp\phi|^2).
    \end{split}
\end{equation*}
Substituting this into the exact equation gives
\begin{equation*}
\phi_{tt}+2\sum_{i=1}^3\phi_i\phi_{ti}
-\bigl(\bar c^2-(\gamma-1)\phi_t+O(|\bp\phi|^2)\bigr)\Delta\phi
+\sum_{i, k=1}^3\phi_i\phi_k\phi_{ik}=0.
\end{equation*}
This implies (\ref{12.30.1.25}).
Note that the effective metric 
\begin{equation}\label{emetric}
   \bg^{00}=\frac{1}{\bar c^2},\, \bg^{0i}=\bg^{0i}=\frac{\p_i \phi}{\bar c^2},\, \bg^{ij} =(\frac{\ga-1}{\bar c^2}\phi_t-1)\delta_{ij}.
\end{equation}
At the ground state $\bp\phi=0$ and $\rho=\bar \rho$,  
define $\bg^{\mu\nu\ga}$ by 
$$
\bg^{\mu\nu\ga}=\frac{\p \bg^{\mu\nu}}{\p(\p_\ga \phi)}\mathlarger|_{\bp\phi=0}.
$$
Now read off the nontrivial components of $g^{\mu\nu\ga}$:
$$g^{0ii}=g^{i0i}=\frac{1}{\bar c^2},\qquad
g^{ii0}=\frac{\gamma-1}{\bar c^2}.
$$
Therefore
\begin{align*}
G(\omega)
&=\sum_{i=1}^3\left(
g^{0ii}\hat\omega_0\hat\omega_i\hat\omega_i
+g^{i0i}\hat\omega_i\hat\omega_0\hat\omega_i
+g^{ii0}\hat\omega_i\hat\omega_i\hat\omega_0
\right)\\
&=\sum_{i=1}^3\left(
\frac1{\bar c^2}(-\bar c)\omega_i^2
+\frac1{\bar c^2}\omega_i(-\bar c)\omega_i
+\frac{\gamma-1}{\bar c^2}\omega_i^2(-\bar c)
\right)=-\frac{\gamma+1}{\bar c}.
\end{align*}

\end{proof}
\subsection{ Geometric calculations for the rotation vector-fields}
In this subsection, we derive geometric identities those are used in proving the comparison estimates in Section \ref{decay}.
\begin{proposition}\label{7.9.1.24}
With ${y'}^k=\bN^k-\frac{x^k}{r}$, there hold
\begin{align}
&1-\l \bN, \p_r\r\approx r^{-2}|\la|^2\label{10.16.3.22}\\
&|y'|\approx  r^{-1}|\la|\label{12.20.3.21}\\
&L\la^a= {}\rp{a}\Omega \Phi-c^{-1}\tensor{\ud \ep}{^a_j_m} v^m \bN^j.\label{3.22.5.21}\\
&L r=c^{-1}v(\bN)+{y'}^k(\frac{ x^k}{r}-c^{-1}v^k)+1.\label{8.17.1.22}\\
\displaybreak[0]
&L (cy^i)+\frac{y^i}{\tir}=k_{A\bN} e_A^i -\frac{v^i}{c \tir}\label{4.7.1.26}
\end{align}
\end{proposition}
\begin{proof}
We first prove
\begin{equation}\label{3.20.2.21}
(\l \bN, \p_r\r)^2+r^{-2}\sum_{a=1}^3(\la^a)^2=1.
\end{equation}
To see it, we decompose
\begin{equation*}
\bN=\l\bN, \p_r\r \p_r+r^{-2}\sum_{i=1}^3\l \bN,{}\rp{i}\O\r {}\rp{i}\O.
\end{equation*}
Note
\begin{align*}
\l\bN, \bN\r&=\l\p_r, \bN\r\l\p_r, \bN\r+r^{-2}\l \bN, {}\rp{i}\O\r\l {}\rp{i}\O, \bN\r\\
&=\l \p_r, \bN\r^2+r^{-2}\sum_{i=1}^3 (\la^i)^2,
\end{align*}
which gives (\ref{3.20.2.21}). By using (\ref{3.20.2.21}), (\ref{10.16.3.22}) is proved.

By definition and (\ref{3.20.2.21}), it is direct to have
$$
\sum_{k=1}^3|{y'}^k|^2=2-2\l \bN, \p_r\r
$$
which gives (\ref{12.20.3.21}) in view of (\ref{10.16.3.22}). 

The proof of (\ref{3.22.5.21}) can be found in \cite[Page 92-93]{Miao_thesis}. 

Next we derive
\begin{equation*}
\bN r={y'}^k \p_k r+\frac{x^k}{r}\p_k r=  {y'}^k \frac{x^k}{r}+1.
\end{equation*}
Using the above identity, by direct calculation, we have
\begin{align*}
L r&=\bT r+\bN r=(cr)^{-1} v^i x^i+1+ {y'}^k \frac{x^k}{r}\\
&=c^{-1}v(\bN)+{y'}^k(\frac{ x^k}{r}-c^{-1}v^k)+1
\end{align*}
which gives (\ref{8.17.1.22}).

Next we prove (\ref{4.7.1.26}). 
Using (\ref{4.9.1.26}), we calculate
\begin{align*}
&(\bd_L\bN)_A=L\bN^i (e_A)_i+L^\a \bN^\b \Ga_{\a\b}^A=L \bN^i (e_A)_i-\sn_A\log c\\
&0=(\bd_L \bN)_\bN=L \bN^i \bN_i + L\log c.
\end{align*}
Then using Lemma \ref{6.7con}, we obtain
\begin{equation*}
L\bN^i=(-\zb_A+\sn_A\log c) e_A^i- L \log c \bN^i.
\end{equation*}
Hence
\begin{align*}
cL(cy^i)&=cL(c\bN^i)-cL(x^i)/\tir-x^i cL(\tir^{-1})\\
&=cL(c\bN^i)-\tir^{-1}\left(c(\bN^i-\frac{x^i}{c\tir})+v^i\right)\\
&=c k_{A\bN}-\tir^{-1}(v^i+c y^i).
\end{align*}
We conclude (\ref{4.7.1.26}).
\end{proof}
\subsection{Derivation of (\ref{4.10.1.19}) and (\ref{4.10.2.19})}\label{Re_wave}
Recall from \cite{{Jared_Luk}} that
\begin{align}
& \Box_{\scg} v^i= \ti\sQ^i, \label{4.10.1.19'}\\
& \Box_{\scg} \varrho= \ti\sQ^0,\label{4.10.2.19'}
\end{align}
with
\begin{align*}
&\ti\sQ^i:=-(1+c^{-1}c')\scg^{\a\b}\p_\a\varrho\p_\b v^i, \\
&\ti\sQ^0:=-3c^{-1}c' \scg^{\a\b}\p_\a \varrho \p_\b \varrho+2\sum_{1\le a<b\le 3}\big(\p_a v^a\p_b v^b-\p_b v^a \p_a v^b\big).
\end{align*}
If $\bg=c^2 \scg$, then it holds for a smooth function $f$ that
\begin{equation}\label{11.21.1.24}
c^2 \Box_\bg f=\Box_{\scg} f+2\scg^{\mu\nu}\p_\mu \log c \p_\nu f.
\end{equation}
We can obtain (\ref{4.10.1.19}) from (\ref{4.10.1.19'}) and the above formula directly.
(\ref{11.21.1.24}) also gives 
\begin{equation}\label{11.21.2.24}
    c^2 \Box_\bg \varrho=\Box_{\scg} \varrho+2\scg^{\mu\nu}\p_\mu\log c \p_\nu \varrho. 
\end{equation}
For the second term in $\ti \sQ^0$,  using (\ref{4.23.1.19}) we calculate directly
\begin{align*}
\sum_{1\le a<b\le 3}\big(\p_a v^a\p_b v^b-\p_b v^a \p_a v^b\big)&=-\f12(\p_i v^j\p_i v^j-(\div v)^2)\\
&=-\f12 (\p_i v^j \p_i v^j-(\bT v^i)^2+(\bT v^i)^2-c^2(\bT \varrho)^2)\\
&=-\f12 (\bg^{\mu\nu}\p_\mu v^j \p_\nu v^j+c^2 \bg^{\mu\nu}\p_\mu \varrho\p_\nu \varrho).
\end{align*}
Substituting the above formula to (\ref{11.21.2.24}) yields (\ref{4.10.2.19}).

\section*{Acknowledgements}

Sergiu Klainerman is supported by  NSF grant DMS-2453843.
Shiwu Yang is supported by the National Key R\&D Program of China 2021YFA1001700 and  the National Science Foundation of China  12141102, 12426203.
Pin Yu is supported by NSFC12141102, New Cornerstone Investigator Program 100001127  and Xiao-Mi Professorship.


\begin{thebibliography}{9999}
\bibitem{Alin_4}  Alinhac, S. {\it Une solution approch\'ee en grand temps des \'equations d’Euler compressibles axisym\'etriques en dimension
deux.} Comm. Partial Differential Equations 17 (1992), no. 3-4, 447–490.
\bibitem{Alinhac}
Alinhac S. {\it Temps de vie des solutions regulieres de equations d'Euler compressibles axisymetriques en 
dimendion deux.} Invent Math, 1993, 111:627-667
\bibitem{Alin1} Alinhac, S. {\it  Blowup of small data solutions for a quasilinear wave equation in two space dimensions.} Ann. of Math.
(2) 149 (1999), no. 1, 97–127.
\bibitem{Alin2} Alinhac, S. {\it Blowup of small data solutions for a class of quasilinear wave equations in two space dimensions. II}, Acta
Math. 182 (1999), no. 1, 1–23.
\bibitem{Andersson_Moncrief} Andersson, L. and Moncrief, V. {\it Elliptic-Hyperbolic Systems and the Einstein Equations.} Ann. Henri Poincar\'e 4, 1–34 (2003).
\bibitem{Buck_2} Buckmaster, T., Shkoller, S. and Vicol, V. {\it Formation of shocks for 2D isentropic compressible Euler.} Comm. Pure
Appl. Math. 75 (2022), no. 9, 2069–2120.
\bibitem{Buck_3} Buckmaster, T., Shkoller, S. and Vicol, V. {\it Formation of point shocks for 3D compressible Euler.} Comm. Pure Appl.
 Math. 76 (2023), no. 9, 2073–2191.
\bibitem{Buck_4} Buckmaster, T., Shkoller, S. and Vicol, V. {\it Shock formation and vorticity creation for 3d Euler.}  Comm. Pure Appl.
 Math. 76 (2023), no. 9, 1965–2072.
\bibitem{CK} Christodoulou, D. and Klainerman, S.,
{\it The Global Nonlinear Stability of Minkowski Space.} Princeton Mathematical Series 41, 1993, 526 pp.
\bibitem{shock_demetrios}Christodoulou, D. {\it The formation of shocks in 3-dimensional fluids.} EMS Monographs in Mathematics, European
 Mathematical Society (EMS), Z\"urich, 2007.
 \bibitem{Miao_thesis} Christodoulou, D. and Miao, S. {\it Compressible Flow and Euler’s Equations}, (monograph, 602 pp.) Surveys in Modern
Mathematics Volume 9, International Press (ISBN 9781571462978), 2014.
\bibitem{Chris_Perez} Christodoulou, D and Perez, D.R. {\it On the formation of shocks of electromagnetic plane waves in
non-linear crystals}, J. Math. Phys. 57, 081506 (2016)
\bibitem{Hormander_1}H\"{o}rmander L. {\it The lifespan of classical solutions of nonlinear hyperbolic equations.} Mittag-Leffler report 
no.5, 1985 

\bibitem{Hormander} H\"{o}rmander, L {\it Nonlinear hyperbolic differential equations.} Lectures, 1986-1987

 \bibitem{FJohn} John, F. {\it  Formation of singularities in one-dimensional nonlinear wave
propagation.} Comm. Pure Appl. Math., 27 (1974) 377-405.

\bibitem{FJohn1}John, F. Blow-up of solutions of nonlinear wave equations in three space dimensions. Manuscripta Math 28, 235–268 (1979).

\bibitem{FJohn2} John, F., {\it Blow up for quasilinear wave equations in three
space dimensions.} Comm. Pure Appl. Math. 34 (1981), 29-51.

   \bibitem{FJohn3} John, F. { \it Blow-up of radial solutions of $ u_{tt} = c^2(u)\triangle u$ in three space dimensions}, Mat. Apl.Comput. 4 (1985), no. 13–18.

 \bibitem{FJohn4} John, F. {\it Existence for large times of strict solutions of nonlinear wave equations in three space dimensions for small initial data}, Comm. Pure Appl. Math. 40 (1987), no. 1, 79–109.


  \bibitem{FJohn5} John  F. {\it Solutions of quasilinear wave equations with small initial data. The third phase}, Nonlinear hyperbolic problems
(Bordeaux, 1988), 1989, pp. 155–184.

\bibitem{John-K} John F.  and Klainerman S. {\it Almost global existence  to nonlinear wave equations in three space dimensions.}
 Comm. Pure Appl. Math. 37, 1984, pp443-455.





\bibitem{K-Ma}  Klainerman,   S. and  Majda  A. {\it Formation of singularities for wave equations including the nonlinear vibrating
string,} Comm. Pure Appl. Math. 33 (1980), no. 3, 241–263. MR562736 (81f:35080)

\bibitem{K1} Klainerman S.  {\it Uniform decay estimates and the Lorentz invariance of the
classical wave equations}, Communications on Pure and Applied
Mathematics, {\bf 38} (1985), 321-332.

\bibitem{K-commu}Klainerman, S.
{\it A commuting vectorfield approach to Strichartz type inequalities and
applications to quasilinear wave equations.}
Int. Math. Res. Notices 2001, No 5, 221--274.

\bibitem{KR1} Klainerman, S. and Rodnianski, I.,
{\it Rough solutions to the Einstein vacuum equations},
Ann. Math. 161 (2005), 1143--1193.
\bibitem{Lax1} Lax, P. {\it Development of singularities of solutions of nonlinear hyperbolic
partial differential equations.} J. Mathematical Phys., 5:5 (1964) 611-
614.
\bibitem{Lax2}Lax, P. {\it The formation and decay of shock waves.} Amer. Math. Monthly, 79, 1972, pp. 227-241.
\bibitem{TPLiu} Liu, T.-P. {\it The development of singularities in the nonlinear waves
for quasi-linear hyperbolic partial differential equations.} J. Differential
Equations, 33 (1979) 92-111.
\bibitem{Jared_Luk} Luk, J and Speck, J. {\it The hidden null struct
equations and a prelude to applications.}  Journal of Hyperbolic Differential Equations, Vol. 17, No. 01, pp. 1-60 (2020)
\bibitem{Spck-luk_2}Luk, J and Speck, J. {\it Shock formation in solutions to the $2D$ compressible Euler equations in the presence of non-zero vorticity.}
 Inventiones mathematicae, 214(3), October 2018.
\bibitem{Spck-luk_3} Luk, J. and Speck, J. {\it The stability of simple plane-symmetric shock formation
for 3D compressible Euler flow with vorticity and entropy.}  Analysis and PDE  Volume 17, No.3, 2024
  \bibitem{Pin-Shuang} Miao, S. and Yu, P. {\it On the formation of shocks for quasilinear wave equations.} Inventiones mathematicae
207 (2017), no. 2, 697–831.

\bibitem{Oleinik}  Oleinik, O.A \textit{  Discontinuous solutions of non-linear differential equations}, Uspehi Mat. Nauk (N.S.) 12 (1957), no. 3(75), 3–73.

\bibitem{Sideris1} Sideris, T. {\it Formation of singularities in solutions to nonlinear hyperbolic equations.}
Arch. Ration. Mech. Anal. 1984-01, Vol.86 (4), p.369-381

\bibitem{Sideris} Sideris, T. {\it Formation of singularities in three-dimensional compressible fluids.} Comm. Math. Phys. 101 (1985), no. 4,
475–485
\bibitem{Spck_shock_1}Speck, J. {\it Shock formation in small-data solutions to 3D quasilinear wave equations.} Mathematical Surveys
and Monographs, 2016.

\bibitem{rough_fluid}Wang, Q. {\it Rough solutions of the 3-D compressible Euler equations.} Ann. of Math. (2) 195 (2022), no. 2,
509–654.
\bibitem{Wang2024} Wang, Q {\it On global dynamics of $3$-D irrotational compressible fluids}, arxiv:2407.13649, 2024,
232pp
\bibitem{Wangnotes} Wang, Q {\it Blowup of small data solutions to $3$-D full Compressible Euler System for a General Polytropic Gas}. In preparation, 2026.  
\bibitem{Yin}Yin, H. and  Qiu, Q. {\it The lifespan for 3-D spherically symmetric compressible euler equations.} Acta Mathematica Sinica 14, 527–534 (1998).

\end{thebibliography}
\end{document}